\documentclass[10pt,journal,compsoc]{IEEEtran}

\usepackage[numbers,sort&compress,square]{natbib}   
\usepackage{graphicx}
\usepackage{subfigure}
\usepackage{amsmath,amssymb}
\usepackage{amsfonts}
\usepackage{array}
\usepackage{amsthm}
\usepackage{mathrsfs}
\usepackage{amsmath}
\usepackage{setspace}
\usepackage{color}
\usepackage[ruled,linesnumbered]{algorithm2e}
\usepackage{epstopdf}
\usepackage{soul}
\usepackage{nnfootnote}
\usepackage{bbm}

\newcommand{\argmax}{\operatornamewithlimits{arg\,max}}
\newcommand{\argmin}{\operatornamewithlimits{arg\,min}}

\renewcommand{\IEEEQED}{\IEEEQEDclosed}
\newtheorem{theorem}{Theorem}
\newtheorem{proposition}[theorem]{Proposition}
\newtheorem{lemma}[theorem]{Lemma}

\theoremstyle{definition}
\newtheorem{define}[theorem]{Definition}

\newtheorem{remark}[theorem]{Remark}

\begin{document}

\title{Scheduling Algorithms for Minimizing Age of Information in Wireless Broadcast Networks with Random Arrivals: The No-Buffer Case}

\author{Yu-Pin Hsu, Eytan Modiano, and Lingjie Duan
	\IEEEcompsocitemizethanks{\IEEEcompsocthanksitem Y.-P. Hsu is with  Department	of Communication Engineering, National Taipei University, Taiwan.\protect\\
		E-mail: yupinhsu@mail.ntpu.edu.tw
		\IEEEcompsocthanksitem E. Modiano is with Laboratory for Information and Decision Systems, Massachusetts Institute of Technology, USA.\protect\\
		E-mail: modiano@mit.edu
		\IEEEcompsocthanksitem L. Duan is with Engineering Systems and Design Pillar, Singapore University of Technology and Design, Singapore.\protect\\
		E-mail: lingjie\_duan@sutd.edu.sg
}
	\thanks{This paper was presented in part in the Proc. of IEEE ISIT, 2017 \cite{hsuage} and 2018 \cite{hsu2018age}.}}

\IEEEtitleabstractindextext{%
	\begin{abstract}
	\textit{Age of information} is a new network performance metric that captures \textit{the freshness of information at end-users}.  This paper studies the age of information from a scheduling perspective. To that end, we consider a wireless broadcast network where a base-station (BS) is updating many users on \textit{random}  information arrivals under a transmission capacity constraint. For the offline case when the arrival statistics are known to the BS, we develop a  \textit{structural MDP scheduling algorithm} and an \textit{index scheduling algorithm},  leveraging Markov decision process (MDP) techniques and the Whittle's methodology for restless bandits. By exploring optimal structural results, we not only reduce the computational complexity of the MDP-based algorithm, but also simplify deriving a closed form of the Whittle index.  Moreover, for the online case, we develop an \textit{MDP-based online scheduling algorithm} and an \textit{index-based online scheduling algorithm}. Both the structural MDP scheduling algorithm and the MDP-based online scheduling algorithm asymptotically minimize the average age, while the index scheduling algorithm  minimizes the average age when the information arrival rates for all users  are the same. Finally, the algorithms are validated via extensive numerical studies.  
	\end{abstract}
	
	}

\maketitle

\IEEEraisesectionheading{\section{Introduction}}
In recent years there has been a growing research interest in an \textit{age of information} \cite{age:kaul}. The age of information is motivated by a variety of network applications requiring  \textit{timely} information. Examples range from information updates for \textit{network users}, e.g.,  live traffic, transportation, air quality, and weather, to status updates for \textit{smart systems}, e.g., smart home systems, smart transportation systems, and smart grid systems.


Fig. \ref{fig:motivation} shows an example network, where  network users $u_1, \cdots, u_N$ are running  applications that need some timely information (e.g., user $u_1$ needs both traffic and transportation information for planning the best route), while at some epochs, snapshots of the information are generated at the sources and sent to the users in the form of packets over wired or wireless networks. The users are being updated and keep the latest information only.   Since the information at the end-users is expected to be as timely as possible, the age of information is therefore proposed to capture the \textit{freshness of the information at the end-users}; more precisely, it measures the elapsed time since the generation of the information.   In addition to the timely information for the network users, the smart systems also need timely status (e.g., locations and velocities in smart transportation systems) to accomplish some tasks (e.g., collision-free smart transportation systems). As such, the age of information is a good metric to evaluate these networks supporting age-sensitive applications. 

\begin{figure}[!t]
\centering
\includegraphics[width=.4\textwidth]{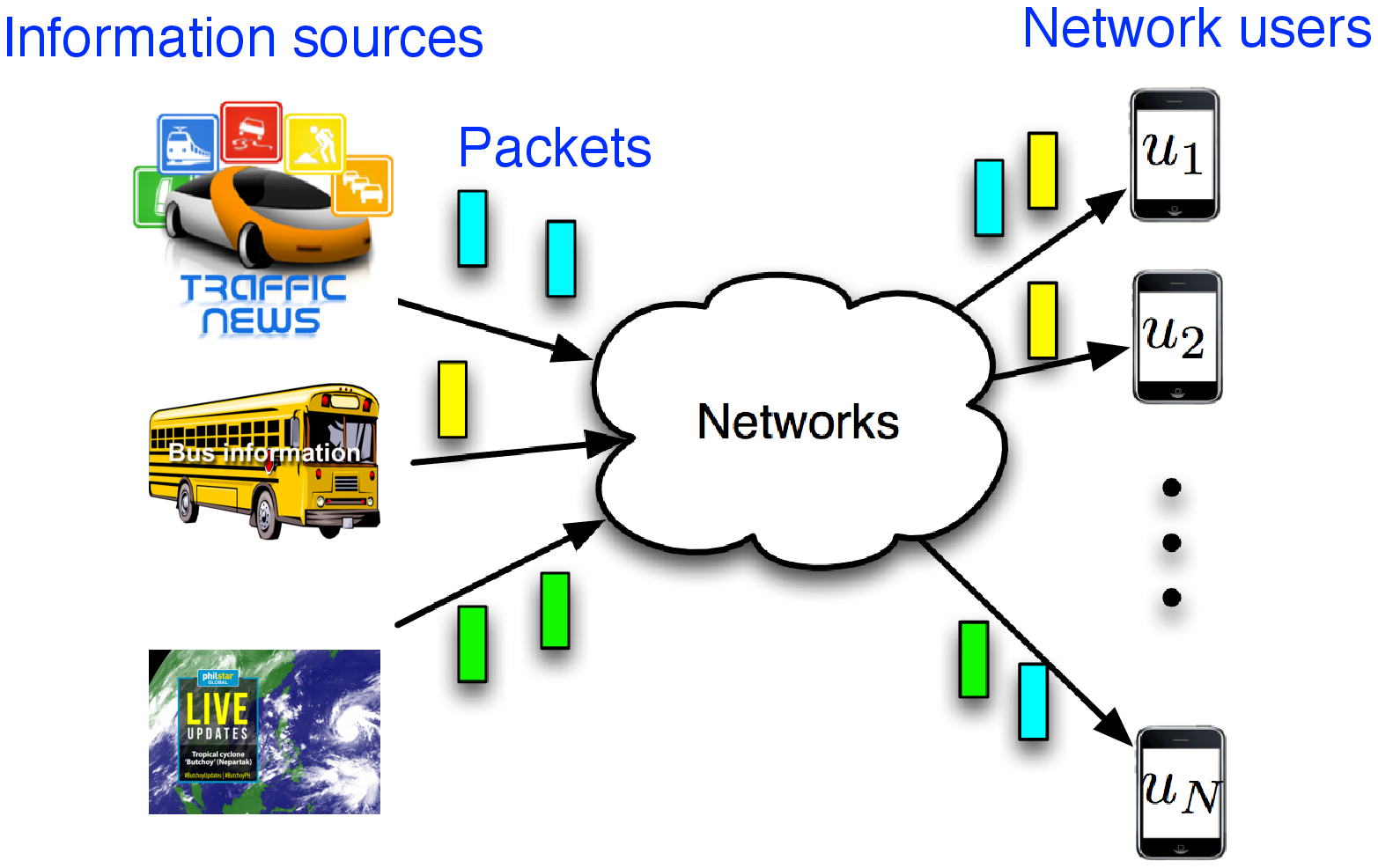}
\caption{Timely information updates for network users.}
\label{fig:motivation}
\end{figure}

%

Next, we characterize the age-sensitive networks in two aspects. First,  while   packet delay is usually referred to as the elapsed time  from the generation to its delivery, the age of information includes not only the packet delay but also the inter-delivery time, because the age of information keeps increasing until the information at the end-users is updated. We hence need to jointly consider the two parameters so as to design an age-optimal network.  Moreover, while  traditional relays (i.e., intermediate nodes) need buffers to keep all packets that are not served yet,  the relays in the network of Fig. \ref{fig:motivation} for timely information at most store the latest information and discard  out-of-date packets. The buffers for minimizing the age here are no longer as useful as those in  traditional relay networks.

 In this paper, we consider a wireless broadcast network, where a base-station (BS) is updating many network users on timely information.  The new information is \textit{randomly} generated at its source. We assume that the BS can serve at most one user for each transmission opportunity. Under the transmission constraint, a transmission scheduling algorithm manages how the channel resources are allocated for each time, depending on the packet arrivals at the BS and the ages of the information at the end-users. The scheduling design is a critical issue to optimize network performance. In this paper we hence develop scheduling algorithms for minimizing the long-run average age.



\subsection{Contributions}
We study the age-optimal scheduling problem in the wireless broadcast network without the buffers at the BS. Our main contributions lie at designing novel  scheduling algorithms and analyzing their age-optimality. For the case when the arrival statistics are available at the BS as prior information, we develop two \textit{offline} scheduling algorithms, leveraging a Markov decision process (MDP) and the Whittle index. However,  the MDP and the Whittle index in our problem will be difficult to analyze as they involve \textit{long-run average cost optimization problems} with \textit{infinite state spaces} and \textit{unbounded immediate costs} \cite{MDP:Bertsekas}. Moreover, it is  in general very challenging to obtain the Whittle index in closed form. By investigating some \textit{structural results},  we not only successfully resolve the issues but also simplify the calculation of the Whittle index. It turns out that our index scheduling algorithm is very simple. When the arrival statistics are unknown, we develop  \textit{online} versions of the two offline algorithms. We show that both offline and online MDP-based scheduling algorithms are asymptotically age-optimal, and the index scheduling algorithm is age-optimal when the information arrival rates for all users are the same. Finally,  we  compare these  algorithms via extensive computer simulations,  and further investigate the impact of the buffers storing the latest information.

\subsection{Related works}
The general idea of  \textit{age} was proposed in \cite{data-age:cho} to study how to refresh a local copy of an autonomous information source to maintain the local copy up-to-date.  The age defined in \cite{data-age:cho} is associated with \textit{discrete} events at the information source, where the age is zero until the  source is updated. Differently, the age of information in \cite{age:kaul} measures the age of a sample of \textit{continuous} events; therefore, the sample immediately becomes old after generated. Many previous works, e.g.,  \cite{age:kaul,age:kam2,age:sun,age:yates2,age:bacinoglu,age:Costa},  studied the age of information  for a single link. The  papers \cite{age:kaul,age:kam2} considered buffers to store all unserved  packets (i.e., out-of-date packets are also stored) and analyzed the long-run average age, based on various  queueing models. They showed that  neither the throughput-optimal sampling rate nor the delay-optimal sampling rate can minimize the average age.  The paper \cite{age:sun} considered a \textit{smart} update  and showed that the \textit{always update} scheme might not minimize the average age.    Moreover, \cite{age:yates2,age:bacinoglu} developed power-efficient updating algorithms for minimizing the average age. The model in \cite{age:Costa} considered no buffer or a buffer to store the latest information.

Of the most relevant works  on scheduling multiple users are \cite{age:he,joowireless,sun2018age,age:igor,yatesage}. The works \cite{age:he,joowireless,sun2018age} considered buffers at a BS to store all out-of-date packets.
The paper \cite{age:igor} considered a buffer to store the latest information with \textit{periodic} arrivals, while information updates in \cite{yatesage} can be generated \textit{at will}. In contrast, our work  is the first to develop both offline and online scheduling algorithms for \textit{random} information arrivals, with the purpose of minimizing the long-run average age. 

\section{System overview} \label{section:system}

\begin{figure}[!t]
\centering
\includegraphics[width=.3\textwidth]{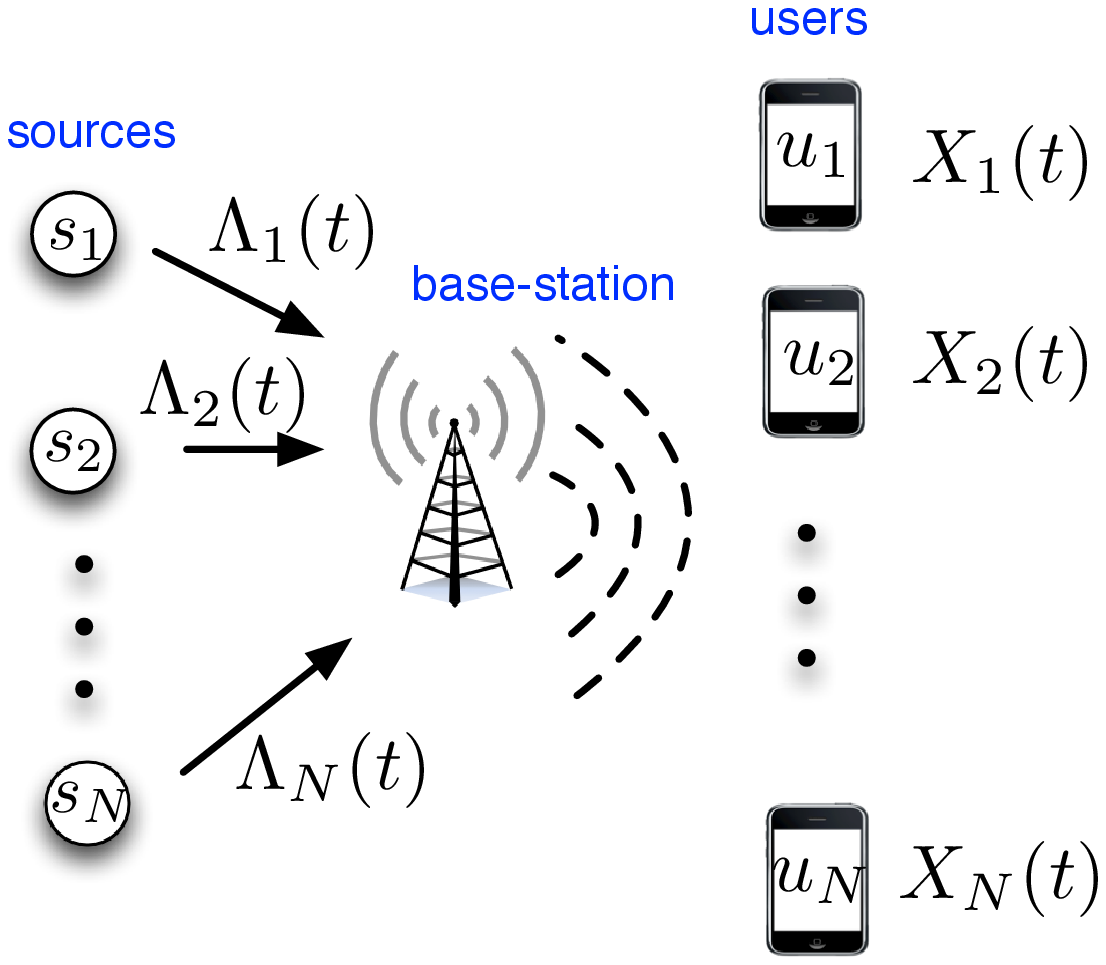}
\caption{A BS  updates $N$ users $u_1, \cdots, u_N$ on  information of sources $s_1, \cdots, s_N$, respectively.}
\label{fig:model}
\end{figure}
\subsection{Network model}
We consider a wireless broadcast network  in Fig.~\ref{fig:model} consisting of a base-station (BS) and $N$ wireless users $u_1, \cdots, u_N$. Each user $u_i$ is interested in timely information generated by a source $s_i$, while the information is transmitted through the BS in the form of packets. We consider a discrete-time system with  slot \mbox{$t=0, 1,  \cdots$}. Packets from the sources (if any) arrive at the BS at the \textit{beginning} of each  slot. The packet arrivals at the BS for different users are independent of each other and also independent and identically distributed (i.i.d.) over slots, following a Bernoulli distribution.  Precisely, by $\Lambda_i(t)$, we indicate if a packet from source $s_i$ arrives at the BS in slot $t$, where $\Lambda_i(t)=1$ if there is a packet; otherwise, $\Lambda_i(t)=0$.  We denote the probability  $P[\Lambda_i(t)=1]$ by $p_i$. 

Suppose  that  the BS can successfully transmit at most one packet during each  slot, i.e., the BS can update at most one user in each slot. By $D(t) \in \{0, 1, \cdots N\}$ we denote a decision of the BS in slot $t$, where $D(t)=0$ if the BS does not transmit any packet and $D(t)=i$ for $i=1, \cdots, N$ if  user $u_i$ is scheduled to be updated in slot $t$.   

In this paper we fosus on the scenario without depolying any buffer at the BS, where an arriving packet is discarded if it is not transmitted in the arriving slot. The \textit{no-buffer network}  is not only simple to implement for practical systems, but also was shown to achieve good performance in a single link (see \cite{age:Costa}). In Section \ref{section:simulation}, we will numerically study  networks with buffers in general.

\subsection{Age of information model}
We initialize the ages of all arriving packets \textit{at the BS} to be zero. The age of information \textit{at a user}  becomes one on receiving a new packet, due to one slot  of the transmission time. Let $X_i(t)$ be the age of information at user $u_i$ in slot $t$ \textit{before} the BS makes a  scheduling decision. Suppose that the age of information  at a user increases linearly with slots if the user is not updated. Then, the dynamics of the age of information for user $u_i$ is
\begin{align}
X_i(t+1)=\left\{
\begin{array}{ll}
1 & \text{if $\Lambda_i(t)=1$ and $D(t)=i$;} \\
X_i(t)+1 & \text{else,}
\end{array}
\right.
\label{eq:age-dynamic}
\end{align}
where the age of information in the next slot is one if the user gets updated on the new information; otherwise, the  age increases by one. Let $\mathbf{X}(t)=(X_1(t), \cdots, X_N(t))$ be the age vector in slot $t$.

Let  $\mathbf{X}$ be the set consisting of all age vectors $(x_1, \cdots, x_N)$ where the ages satisfy $x_i \geq 1$ for all $i$ and $x_i\neq x_j$ for all $i \neq j$. Since the BS can update at most one user in each slot, if an initial age vector $\mathbf{X}(0)$ is outside   $\mathbf{X}$, then eventually  age vector $\mathbf{X}(t)$ will enter  $\mathbf{X}$ and  stay in  $\mathbf{X}$ onwards; otherwise, someone is never updated and its age approaches infinity. In other words, the age vector outside $\mathbf{X}$ is \textit{transient}. Without loss of generality, we  assume that initial age vector $\mathbf{X}(0)$ is within  $\mathbf{X}$. Later, in the proof of Lemma~\ref{thm:stationary}, we will show that any transmission decision before the age vector enters  $\mathbf{X}$ will not affect the minimum \textit{average age} (defined in the next section). 



\subsection{Problem formulation}
A \textit{scheduling algorithm} $\pi=\{D(0), D(1), \cdots\}$ specifies a transmission decision for each slot.   We  define the \textit{average age} under   scheduling algorithm $\pi$  by
\begin{align*}
\limsup_{T \rightarrow \infty} \frac{1}{T+1} E_{\pi} \left[ \sum_{t=0}^T \sum_{i=1}^N X_i(t)\right], 
\end{align*}
where $E_{\pi}$ represents the conditional expectation, given that scheduling algorithm $\pi$ is employed. We remark that this paper focuses on the total age of users for delivering clean results; whereas our design and analysis can  work perfectly for the \textit{weighted sum} of  the ages. Our goal is  to develop \textit{age-optimal} scheduling algorithms, defined below.  

\begin{define}
A scheduling algorithm $\pi$ is \textit{age-optimal} if it minimizes the average age.
\end{define}

In this paper, we will develop two offline scheduling algorithms and two online scheduling algorithms. Leveraging Markov decision process (MDP) techniques and Whittle's methodology, we  develop two offline scheduling algorithms in Sections~\ref{section:mdp} and \ref{section:whittle}, respectively, when the arrival statistics are available to the BS; later, two  online versions of the  offline  algorithms are proposed in Section~\ref{section:online} for the case when the arrival statistics are oblivious to the BS. 


\section{A structural MDP scheduling algorithm} \label{section:mdp}
Our first scheduling algorithm is driven by the MDP techniques. To that end, we formulate our problem as an MDP $\Delta$ with the  components \cite{MDP:Puterman} below. 
\begin{itemize}
\item \textbf{States}: We define the state $\mathbf{S}(t)$ of the MDP in slot $t$  by $\mathbf{S}(t)=(X_1(t), \cdots, X_N(t),$ $\Lambda_1(t), \cdots, \Lambda_N(t))$. Let $\mathbf{S}$ be the state space consisting of all states $(x_1, \cdots, x_N, \lambda_1, \cdots, \lambda_N)$ where  
\begin{itemize}
	\item $(x_1, \cdots, x_N) \in \mathbf{X}$ or $x_i> N$   for all $i$;
	\item  $\lambda_i \in \{0,1\}$ for all $i$.  
\end{itemize}
The state space includes some transient age vectors. That is used to fit \textit{truncated states} in Section~\ref{subsection:finite-state}. We will show later in Lemma~\ref{thm:stationary} that adding these transient states will not change the minimum average age. Note that $\mathbf{S}$ is a \textit{countable infinite} set because the ages are possibly unbounded.

\item \textbf{Actions}: We define the action of the MDP in slot $t$ to be $D(t)$.
Note that the action space is finite.

\item \textbf{Transition probabilities}: 
By $P_{\mathbf{s},\mathbf{s}'}(d)$ we denote the transition probability of the MDP from state $\mathbf{s}=(x_1, \cdots, x_N, \lambda_1, \cdots, \lambda_N)$ to state $\mathbf{s}'=(x_1', \cdots, x_N',$ $\lambda_1', \cdots, \lambda_N')$ under  action $D(t)=d$. According to the age dynamics in Eq. (\ref{eq:age-dynamic}) and the i.i.d. assumption of the arrivals, we can describe the non-zero $P_{\mathbf{s},\mathbf{s}'}(d)$ as
\begin{align*}
P_{\mathbf{s},\mathbf{s}'}(d)=
\prod_{i:\lambda'_i=1} p_i \prod_{i:\lambda'_i=0} (1-p_i), 
\end{align*}
if $x_i'=(x_i+1)-x_i\mathbbm{1}_{i =d \text{\,\,and\,\,} \lambda_i= 1}$ for all $i=1, \cdots, N$, where $\mathbbm{1}$ is the indicator function. 

\item \textbf{Cost}: Let $C(\textbf{S}(t), D(t)=d)$ be the \textit{immediate cost} of the MDP if action $D(t)=d$ is taken in slot $t$ under state $\mathbf{S}(t)$, representing the resulting total age in the next slot:
\begin{align*}
C(\textbf{S}(t), D(t)=d) \triangleq & \sum^N_{i=1} X_i(t+1)\\
=&\sum_{i=1}^N (X_i(t)+1)-X_d(t)\cdot \Lambda_d(t), 
\end{align*} 
where we define  $X_0(t)=0$ and $\Lambda_0(t)=0$ for all $t$ (for the no update case  $d=0$), while the last term indicates  that user $u_d$  is updated in slot $t$. 

\end{itemize}

%

The objective of the MDP $\Delta$ is to find a \textit{policy} $\pi$  (with the same definition as the scheduling algorithm) that minimizes the \textit{average cost} $V(\pi)$ defined by
\begin{align*}
V(\pi) = \limsup_{T \rightarrow \infty} \frac{1}{T+1} E_{\pi} \left[ \sum_{t=0}^T C(\mathbf{S}(t),D(t))\right].
\end{align*}

\begin{define}
A policy $\pi$ of the MDP $\Delta$ is \textit{$\Delta$-optimal}  if it minimizes the average cost $V(\pi)$. 
\end{define}
Then, a $\Delta$-optimal policy is an age-optimal scheduling algorithm. Moreover, policies of the MDP can be classified as follows.  
A policy of the MDP is \textit{history dependent} if $D(t)$ depends on $D(0), \cdots, D(t-1)$ and $\mathbf{S}(0) \cdots, \mathbf{S}(t)$.  A policy is \textit{stationary} if $D(t_1)=D(t_2)$ when $\mathbf{S}(t_1)=\mathbf{S}(t_2)$ for any $t_1, t_2$.    A \textit{randomized} policy specifies a probability distribution on the set of decisions, while a \textit{deterministic} policy makes a decision with certainty. A policy in general  belongs to one of the following sets \citep{MDP:Puterman}:
\begin{itemize}
	\item  $\Pi^{\text{HR}}$: a set of randomized history dependent policies;
	\item $\Pi^{\text{SR}}$: a set of  randomized stationary policies;
	\item $\Pi^{\text{SD}}$: a set of  deterministic stationary policies.
\end{itemize}
Note  that $\Pi^{\text{SD}} \subseteq \Pi^{\text{SR}} \subseteq \Pi^{\text{HR}}$ \cite{MDP:Puterman}, while the  complexity of searching a $\Delta$-optimal policy increases from left to right.  According to \cite{MDP:Puterman}, there may exist neither $\Pi^{\text{SR}}$ nor $\Pi^{\text{SD}}$ policy that is $\Delta$-optimal. Hence, we target at exploring a regime under which a $\Delta$-optimal policy lies in a smaller policy set $\Pi^{\text{SD}}$, and investigating  its structures.

%
%
%

\subsection{Characterization of the $\Delta$-optimality}  \label{section:characterization}
To characterize the $\Delta$-optimality, we start with introducing an  infinite horizon \textit{$\alpha$-discounted cost}, where $0 < \alpha < 1$ is a discount factor. We subsequently connect the discounted cost case to the average cost case, because structures of a $\Delta$-optimal policy   usually depend on its discounted cost case.

Given  initial state $\mathbf{S}(0)=\mathbf{s}$, the \textit{expected total $\alpha$-discounted cost} under  scheduling algorithm $\pi$ (that can be history dependent) is
\begin{align*}
V_{\alpha}(\mathbf{s}; \pi) =\lim_{T \rightarrow \infty} E_\pi \left[ \sum_{t=0}^T \alpha^t C(\mathbf{S}(t), D(t)) |\mathbf{S}(0)=\mathbf{s}\right]. 
\end{align*}

\begin{define}
A policy $\pi$ of the MDP $\Delta$ is \textit{$\Delta_{\alpha}$-optimal} if it minimizes the expected total $\alpha$-discounted cost $V_{\alpha}(\mathbf{s}; \pi)$. In particular, we define
\begin{align*}
V_{\alpha}(\mathbf{s})=\min_{\pi}V_{\alpha}(\mathbf{s};
\pi).
\end{align*}
\end{define}
Moreover, by $h_{\alpha}(\mathbf{s})=V_{\alpha}(\mathbf{s})-V_{\alpha}(\mathbf{0})$ we define the \textit{relative cost function}, which is the difference of the minimum discounted costs between state $\mathbf{s}$ and a reference state $\mathbf{0}$. We can arbitrarily choose the reference state, e.g., $\mathbf{0}=(1,2,\cdots,N, 1, \cdots, 1)$ in this paper. We then introduce the \textit{discounted cost optimality equation} of $V_{\alpha}(\mathbf{s})$  below.

\begin{proposition} \label{lemma:optimality-eq}
The minimum expected total $\alpha$-discounted cost $V_{\alpha}(\mathbf{s})$, for initial state $\mathbf{s}$, satisfies the following \textit{discounted cost optimality equation}:
\begin{align}
V_\alpha(\mathbf{s}) &= \min_{d \in \{0,1, \cdots, N\}}  C(\mathbf{s},d) + \alpha E[V_\alpha(\mathbf{s}') ], \label{eq:discounte-optimality-equation}
\end{align}
where the expectation is taken over all possible next state $\mathbf{s}'$ reachable from the state $\mathbf{s}$, i.e., $E[V_\alpha(\mathbf{s}')]=\sum_{\mathbf{s}' \in \mathbf{S}} P_{\mathbf{s},\mathbf{s}'}(d)V_{\alpha}(\mathbf{s}')$.
A deterministic stationary policy that realizes the minimum of the right-hand-side (RHS) of  the discounted cost optimality equation in Eq. (\ref{eq:discounte-optimality-equation})  is a $\Delta_{\alpha}$-optimal policy. Moreover, we can define a value iteration $V_{\alpha, n}(\mathbf{s})$ by $V_{\alpha, 0}(\mathbf{s})= 0$ and  for any $n \ge 0$,
\begin{align}
V_{\alpha, n+1}(\mathbf{s}) &= \min_{d \in \{0,1, \cdots, N\}} C(\mathbf{s},d) + \alpha E[V_{\alpha, n}(\mathbf{s}') ]. \label{eq:discount-itr}
\end{align}
Then, 
$V_{\alpha, n}(\mathbf{s}) \rightarrow V_{\alpha}(\mathbf{s})$ as $n \rightarrow \infty$,  for every $\mathbf{s}$ and $\alpha$.
\end{proposition}

\begin{proof}
Please see Appendix \ref{appendix:lemma:optimality-eq}.
\end{proof}

The value iteration in Eq. (\ref{eq:discount-itr}) is helpful for characterizing $V_{\alpha}(\mathbf{s})$, e.g., showing that   $V_{\alpha}(\mathbf{s})$   is a non-decreasing function in the following.

\begin{proposition}  \label{lemma:monotone}
$V_{\alpha}(x_i, \mathbf{x}_{-i}, \boldsymbol{\lambda})$ is a non-decreasing function in $x_i$, for  given $\mathbf{x}_{-i}=(x_1, \cdots, x_N) - \{x_i\}$ and $\boldsymbol{\lambda}=(\lambda_1, \cdots, \lambda_N)$.
\end{proposition}
\begin{proof}
Please see Appendix \ref{appendix:lemma:monotone}.
\end{proof}

Using  Propositions~\ref{lemma:optimality-eq} and \ref{lemma:monotone} for the discounted cost case, we show that the MDP $\Delta$  has a deterministic stationary $\Delta$-optimal policy, as follows.

\begin{lemma} \label{thm:stationary}
There exists a  deterministic stationary policy  that is $\Delta$-optimal. 
Moreover, there exists a finite constant $V^*=\lim_{\alpha \rightarrow 1}(1-\alpha)V_{\alpha}(\mathbf{s})$ for every state $\mathbf{s}$ such that the minimum average  cost is $V^*$, independent of  initial state $\mathbf{s}$. 
\end{lemma}
\begin{proof}
Please see Appendix \ref{appendix:thm:stationary}.
\end{proof}

We want to further elaborate  on Lemma~\ref{thm:stationary}. 
\begin{itemize}
	\item First, note that there is no condition for the existence of a deterministic stationary policy that is $\Delta$-optimal. In general, we need some  conditions to ensure that the  reduced Markov chain by a  deterministic stationary policy is positive recurrent. Intuitively, we can think of the age of our problem as an age-queuing system,  consisting of an age-queue, input to the queue, and a server. The input rate is one per slot since the age increases by one for each slot, while the server can serve an \textit{infinite} number of age-packets for each service opportunity. As such, we always can find a scheduling algorithm such that the average arrival rate is  less than  the service rate and thus the reduced Markov chain is positive recurrent. Please see the proof in Appendix~\ref{appendix:thm:stationary} for details.
	\item Second, since our MDP $\Delta$ involves a \textit{long-run average cost optimization} with a \textit{countably infinite} state space and \textit{unbounded} immediate cost, a $\Delta$-optimal policy of such an MDP might not satisfy the  average cost optimaility equation like Eq. (\ref{eq:discounte-optimality-equation}) (see \cite{cavazos1991counterexample} for a counter-example), even though the optimality of a deterministic stationary policy is established in Lemma \ref{thm:stationary}. 
\end{itemize}

In addition to the optimality of  deterministic stationary policies, we show that a $\Delta$-optimal policy  has a  nice structure. To investigate such structural results not only facilitates the scheduling algorithm  design in Section~\ref{section:mdp}, but also simplifies the calculation of the Whittle index in Section~\ref{section:whittle}. 

\begin{define}
A \textit{switch-type} policy  is a special deterministic stationary policy of the MDP $\Delta$: for  given $\mathbf{x}_{-i}$ and $\boldsymbol{\lambda}$, if the action of the policy for state $\mathbf{s}=(x_i, \mathbf{x}_{-i}, \boldsymbol{\lambda})$ is $d_{\mathbf{s}}=i$, then the action for state $\mathbf{s}'=(x_i+1,\mathbf{x}_{-i},\boldsymbol{\lambda})$ is $d_{\mathbf{s}'}=i$ as well. 
\end{define}

In general, showing that a $\Delta$-optimal policy satisfies a structure relies on an optimality equation; however, as  discussed, the average cost optimality equation for the MDP $\Delta$ might not be available. To resolve this issue, we first investigate the discounted cost case by the well-established value iteration in Eq.~(\ref{eq:discount-itr}), and then extend to the average cost case. 

\begin{theorem} \label{theorem:optimal-switch}
There exists a  switch-type policy of the MDP $\Delta$ that is $\Delta$-optimal.
\end{theorem}
\begin{proof}
First, we start with the discounted cost case, and show that a $\Delta_{\alpha}$-optimal scheduling algorithm is  switch-type.  Let $\nu_{\alpha}(\mathbf{s};d)=C(\mathbf{s},d) + \alpha E[V_\alpha(\mathbf{s}')]$. Then, $V_{\alpha}(\mathbf{s})=\min_{d \in \{0, 1, \cdots, N\}} \nu_{\alpha}(\mathbf{s};d)$. Without loss of generality, we suppose that a $\Delta_{\alpha}$-optimal action at state $\mathbf{s}=(x_1,\mathbf{x}_{-1}, \boldsymbol{\lambda})$ is to update the user $u_1$ with $\lambda_1=1$. Then, according to the optimality of $d^*_{(x_1, \mathbf{x}_{-1}, \boldsymbol{\lambda})}=1$,
\begin{align*}
\nu_{\alpha}(x_1,\mathbf{x}_{-1},\boldsymbol{\lambda};1) - \nu_{\alpha}(x_1,\mathbf{x}_{-1},\boldsymbol{\lambda};j) \leq 0, 
\end{align*}
for all $j \neq 1$.

Let $\mathbf{1}=(1, \cdots, 1)$ be the vector with all entries being one. Let $\mathbf{x}_i=(0, \cdots, x_i, \cdots,0)$ be the zero vector except for the $i$-th entry being replaced by $x_i$.  To demonstrate the switch-type structure, we consider the following two cases:
\begin{enumerate}
	\item \textit{For any other user $u_j$ with $\lambda_j=1$}: Since $V_{\alpha}(x_1, \mathbf{x}_{-1},\boldsymbol{\lambda})$ is a non-decreasing function in $x_1$ (see Proposition \ref{lemma:monotone}),  we have  
\begin{align*}
&\nu_{\alpha}(x_1+1, \mathbf{x}_{-1},\boldsymbol{\lambda};1) - \nu_{\alpha}(x_1+1, \mathbf{x}_{-1},\boldsymbol{\lambda};j) \\
= & x_j-(x_1+1)+\alpha E[V_{\alpha}(1, \mathbf{x}_{-1}+\mathbf{1},\boldsymbol{\lambda}')\\
&- V_{\alpha}(x_1+2, \mathbf{x}_{-1}+\mathbf{1}-\mathbf{x}_j,\boldsymbol{\lambda}') ]\\
\leq & x_j-x_1+\alpha E[V_{\alpha}(1, \mathbf{x}_{-1}+\mathbf{1},\boldsymbol{\lambda}')\\
&- V_{\alpha}(x_1+1, \mathbf{x}_{-1}+\mathbf{1}-\mathbf{x}_j,\boldsymbol{\lambda}') ]\\
= & \nu_{\alpha}(x_1, \mathbf{x}_{-1},\boldsymbol{\lambda};1) - \nu_{\alpha}(x_1, \mathbf{x}_{-1},\boldsymbol{\lambda};j) \leq 0,
\end{align*} 
where $\boldsymbol{\lambda}'$ is the next arrival vector. 
	\item \textit{For any other user $u_j$ with $\lambda_j=0$}: Similarly, we have
	\begin{align*}
&\nu_{\alpha}(x_1+1, \mathbf{x}_{-1},\boldsymbol{\lambda};1) - \nu_{\alpha}(x_1+1, \mathbf{x}_{-1},\boldsymbol{\lambda};j) \\
= &-(x_1+1)+\alpha E[V_{\alpha}(1, \mathbf{x}_{-1}+\mathbf{1},\boldsymbol{\lambda}')\\
&- V_{\alpha}(x_1+2, \mathbf{x}_{-1}+\mathbf{1},\boldsymbol{\lambda}') ]\leq 0.
\end{align*} 
\end{enumerate}
Considering the two cases,  a $\Delta_{\alpha}$-optimal action for state $(x_1+1, \mathbf{x}_{-1},\boldsymbol{\lambda})$ is still to update $u_1$, yielding the switch-type structure.

Then, we preceed to establish the optimality for the average cost case. Let $\{\alpha_n\}_{n=1}^{\infty}$ be a sequence of the discount factors. According to \cite{stationary-policy:Sennott}, if the both conditions in Appendix~\ref{appendix:thm:stationary} hold, then there exists a subsequence $\{\beta_n\}_{n=1}^{\infty}$ such that a $\Delta$-optimal algorithm is the limit point of the $\Delta_{\beta_n}$-optimal policies. By induction on $\beta_n$ again, we obtain that a $\Delta$-optimal is  switch-type as well.
\end{proof}

\subsection{Finite-state MDP approximations}\label{subsection:finite-state}
The classical method for solving an MDP is to apply a value iteration method \cite{MDP:Puterman}. However, as mentioned, the average cost optimality equation might not exist. Even though average cost value iteration holds like Eq.~(\ref{eq:discount-itr}), the value iteration cannot work in practice, as we need to update an infinite number of states for each iteration. To address the issue, we propose a sequence of finite-state approximate MDPs. In general, a sequence of approximate MDPs might not converge to the original MDP according to \cite{MDP:Sennott}. Thus, we will rigorously show the convergence of the proposed sequence.

Let $X^{(m)}_i(t)$ be a \textit{virtual age} of information for user $u_i$ in slot $t$, with the dynamic being 
\begin{align*}
X^{(m)}_{i}(t+1)=\left\{
\begin{array}{ll}
1 & \text{if $\Lambda_i(t)=1$, $D(t)=i$} ;\\
\left[X^{(m)}_i(t)+1\right]^+_m & \text{else},
\end{array}
\right.
\end{align*}
where we define the notation $[x]^+_m$ by $[x]^+_m=x$ if $x \leq m$ and $[x]^+_m=m$ if $x >m$, i.e., we truncate the real age by $m$. This is different from Eq.~(\ref{eq:age-dynamic}). While the real age $X_i(t)$ can go beyond $m$, the virtual age $X^{(m)}_i(t)$ is at most $m$. Here, we reasonably choose the truncation $m$ to be greater than the number $N$ of users, i.e., $m > N$. Later, in Appendix~\ref{appendix:theorem:finite-approximation} (see Remark~\ref{remark to truncation}), we will discuss some mathematical reasons for the choice.

By $\{\Delta^{(m)}\}_{m=N+1}^{\infty}$ we define a sequence of approximate MDPs for $\Delta$, where each MDP $\Delta^{(m)}$ is the same as the original MDP $\Delta$  except:
\begin{itemize}
	\item \textbf{States}: The state  in slot $t$ is  $S^{(m)}(t)=(X^{(m)}_1(t),$ $\cdots, X^{(m)}_N(t), \Lambda_1(t), \cdots,\Lambda_N(t))$. Let $\mathbf{S}^{(m)}$ be the state space.
	\item \textbf{Transition probabilities}: Under  action $D(t)=d$, the transition probability $P^{(m)}_{\mathbf{s},\mathbf{s}'}(d)$	of the MDP $\Delta^{(m)}$ from state $\mathbf{s}=(x_1, \cdots, x_N, \lambda_1, \cdots, \lambda_N)$ to state $\mathbf{s}'=(x_1', \cdots, x_N',$ $\lambda_1', \cdots, \lambda_N')$ is
\begin{align*}
P_{\mathbf{s},\mathbf{s}'}(d)=
\prod_{i:\lambda'_i=1} p_i \prod_{i:\lambda'_i=0} (1-p_i), 
\end{align*}
if $x_i'=[(x_i+1)-x_i\mathbbm{1}_{i =d \text{\,\,and\,\,} \lambda_i= 1}]^+_m$ for all $i=1, \cdots, N$.
\end{itemize}
Remember that the state space $\mathbf{S}$ of the MDP $\Delta$ includes some transient age vectors, e.g., $(N, \cdots, N)$. That is because, if not, the truncated state space $\mathbf{S}^{(m)}$ would  not be a subset of original state space $\mathbf{S}$.

Next, we  show that  the proposed sequence of approximate MDPs converges to the $\Delta$-optimum. 

\begin{theorem} \label{theorem:finite-approximation}
Let $V^{(m)*}$ be the minimum average cost for the MDP  $\Delta^{(m)}$. Then, $V^{(m)*} \rightarrow V^*$ as $m \rightarrow \infty$.
\end{theorem}
\begin{proof}
Please see Appendix \ref{appendix:theorem:finite-approximation}.
\end{proof}

\subsection{Structural MDP scheduling algorithm}
Now, for a given truncation $m$, we are ready to propose a practical algorithm to solve the MDP $\Delta^{(m)}$. The traditional \textit{relative value iteration algorithm} (RVIA), as follows, can be applied to obtain an optimal deterministic stationary  policy for $\Delta^{(m)}$:
\begin{align}
V^{(m)}_{n+1}(\mathbf{s}) = \min_{d \in \{0,1, \cdots, N\}} C(\mathbf{s},d) +  E[V^{(m)}_{n}(\mathbf{s}') ] - V^{(m)}_n(\mathbf{0}),  \label{eq:rvia}
\end{align}
for all $\mathbf{s} \in \mathbf{S}^{(m)}$ where the initial value function is $V^{(m)}_0(\mathbf{s})=0$. 
For each iteration, we need to update  actions for \textit{all} virtual states by minimizing the RHS of Eq.~(\ref{eq:rvia}) as well as update  $V^{(m)}(\mathbf{s})$ for \textit{all} $\mathbf{s} \in \mathbf{S}^{(m)}$. As the size of the state space is $O(m^N)$,  the computational complexity of updating all virtual states in each iteration of  Eq. (\ref{eq:rvia})  is more than $O(m^N)$. The complexity primarily results from the  truncation $m$ of the  MDP $\Delta^{(m)}$ and the number $N$ of users. In this section, we focus on dealing with large values of $m$ for the case of fewer users. In next section we will solve the case of more users. 

To develop a low-complexity scheduling algorithm for fewer users, we propose  \textit{structural RVIA} in Alg. \ref{alg:offline}, which is an improved RVIA by leveraging the switch-type structure. In Alg. \ref{alg:offline}, we seek an optimal action $d^*_{\mathbf{s}}$ for each virtual state $\mathbf{s} \in \mathbf{S}^{(m)}$ by iteration. For each iteration, we update both the optimal action $d^*_{\mathbf{s}}$ and  $V^{(m)}(\mathbf{s})$ for all virtual states. If the switch property holds\footnote{The optimal policy for the truncated MDPs is  switch-type as well, according to the same proof as Theorem~\ref{theorem:optimal-switch}.}, we can determine an optimal action \textit{immediately} in Line \ref{alg:offline-switch}; otherwise we find an optimal action according to Line \ref{alg:offline-regular-update}. By $V_{\text{tmp}}(\mathbf{s})$ in Line~\ref{alg:offline-value-update-tmp} we temporarily keep the updated value, which will  replace $V^{(m)}(\mathbf{s})$ in Line~\ref{alg:offline-value-update}. Using the switch structure to prevent from the minimum operations on all  virtual  states  in the conventional RVIA, we can reduce the computational complexity resulting from the size $m$.  Next, we establish the optimality of the structural RVIA for the approximate MDP $\Delta^{(m)}$. 
\begin{theorem} \label{theorem:truncation}
For  MDP $\Delta^{(m)}$ with a given $m$, the limit point of $d^*_{\mathbf{s}}$ in Alg.~\ref{alg:offline} is a $\Delta^{(m)}$-optimal action for every virtual  state $\mathbf{s} \in \mathbf{S}^{(m)}$.  In particular, Alg.~\ref{alg:offline} converges to the $\Delta^{(m)}$-optimum in a finite number of iterations.  
\end{theorem}
\begin{proof}
(Sketch) According to \cite[Theorem 8.6.6]{MDP:Puterman}, we only need to verify that the truncated MDP is \textit{unichain}. Please see Appendix \ref{appendix:theorem:truncation} for details. 
\end{proof}

\begin{algorithm}[!t]
\SetCommentSty{text}
\SetAlgoLined 
\SetKwFunction{Union}{Union}\SetKwFunction{FindCompress}{FindCompress} \SetKwInOut{Input}{input}\SetKwInOut{Output}{output}

$V^{(m)}(\mathbf{s}) \leftarrow 0$ for all virtual states $\mathbf{s}\in \mathbf{S}^{(m)}$;\\

\While{$1$}{
\ForAll{$\mathbf{s} \in \mathbf{S}^{(m)}$}{
\uIf{there exists $\zeta >0$ and $i \in \{1, \cdots, N\}$ such that $d^*_{(x_i-\zeta,\mathbf{x}_{-i},\boldsymbol{\lambda})}=i$}{
$d^*_{\mathbf{s}} \leftarrow i$;\\ \label{alg:offline-switch}
}
\Else{
$d^*_{\mathbf{s}} \leftarrow \argmin_{d \in \{0, \cdots, N\}} C(\mathbf{s},d)+E[V^{(m)}(\mathbf{s}')]$;\\ \label{alg:offline-regular-update}
}
$V_{\text{tmp}}(\mathbf{s}) \leftarrow C(\mathbf{s},d^*_{\mathbf{s}})+E[V^{(m)}(\mathbf{s}')]-V^{(m)}(\mathbf{0})$;\\ \label{alg:offline-value-update-tmp}
}
$V^{(m)}(\mathbf{s}) \leftarrow V_{\text{tmp}}(\mathbf{s})$ for all $\mathbf{s} \in \mathbf{S}^{(m)}$.\\ \label{alg:offline-value-update}
}

\caption{Structural RVIA.}
\label{alg:offline}
\end{algorithm}

Based on the structural RVIA in Alg.~\ref{alg:offline}, we propose the \textit{structural MDP scheduling algorithm}: Given the actions for all state $\mathbf{s} \in \mathbf{S}^{(m)}$ from Alg.~\ref{alg:offline},  for each slot $t$ the  scheduling algorithm makes a decision according to the \textit{virtual age} $X^{(m)}_i(t)$ for all $i$,   instead of the \textit{real age} $X_i(t)$. Then, combining Theorems~\ref{theorem:finite-approximation} and \ref{theorem:truncation} yields that the proposed algorithm is asymptotically $\Delta$-optimal as $m$ approaches infinity.

\section{An index scheduling algorithm} \label{section:whittle}
By mean of the MDP techniques, we have developed the structural MDP scheduling algorithm. The scheduling algorithm not only reduces the complexity from the traditional RVIA, but also was shown to be asymptotically age-optimal. 
However,  the scheduling algorithm might not be feasible for many users; thus, a low-complexity scheduling algorithm for many users is still needed. To fill this gap, we  investigate the scheduling problem from the perspective of \textit{restless bandits} \cite{gittins2011multi}.  A restless bandit generalizes a classic bandit  by allowing the bandit to keep evolving under a \textit{passive} action, but in a distinct way from its  continuation under an \textit{active} action. 

The  restless  bandits problem, in general, is PSPACE-hard \cite{gittins2011multi}. Whittle hence investigated a relaxed version, where a constraint on the number of active bandits for each slot is replaced by the expected number.  With this relaxation, Whittle then applied a Lagrangian approach to decouple the multi-armed bandit  problem into multiple sub-problems, while proposing an index policy and a concept of \textit{indexability}. The index policy is optimal for the relaxed problem; moreover, in many practical systems, the low-complexity index policy performs remarkably well, e.g., see \cite{larranaga2015stochastic}.  

With the success of the Whittle index policy to solve the restless bandit problem, we  apply the  Whittle's approach to develop a low-complexity scheduling algorithm. However, to obtain the Whittle index in closed form and to establish the indexability can be very challenging \cite{gittins2011multi}. To address the issues, we simplify the derivation of the Whittle index by investigating structural results like Section.~\ref{section:mdp}.

\subsection{Decoupled sub-problem}
We note that each user in our scheduling problem can be viewed as a restless bandit. Then, applying the Whittle's approach, we can decouple our problem into $N$ sub-problems. Each sub-problem consists of a single user $u_i$ and adheres to the network model in Section \ref{section:system} with $N=1$, except for an additional cost $c$ for updating the user. In fact, the cost $c$ is a scalar Lagrange multiplier in the Lagrangian approach. In each decoupled sub-problem, we aim at  determining whether or not the BS  updates  the user in each slot, for striking a balance between the updating cost and the cost incurred by age.  Since each sub-problem consists of a single user only, hereafter in this section we omit the index $i$ for simplicity.

Similarly, we  cast the sub-problem into an MDP $\Omega$, which is the same as the MDP $\Delta$ in Section~\ref{section:mdp} with a single user except:
\begin{itemize}
	\item \textbf{Actions}: Let $A(t) \in \{0, 1\}$ be an action of the MDP in slot $t$ indicating the BS's decision, where $A(t)=1$ if the BS decides \textit{to update} the user and $A(t)=0$ if the BS decides \textit{to idle}. Note that the action $A(t)$ is different from the scheduling decision $D(t)$. The action $A(t)$ is used for the decoupled sub-problem.  In Section.~\ref{subsection:index}, we will use the MDP $\Omega$ to decide $D(t)$. 
	\item \textbf{Cost}: Let $C(\mathbf{S}(t), A(t))$ be an immediate cost if action  $A(t)$ is taken in slot $t$ under  state $\mathbf{S}(t)$, with the definition as follows.
\begin{align}
&C\Bigl(\textbf{S}(t)=(x, \lambda), A(t)=a\Bigr) \nonumber\\
\triangleq& (x+1- x \cdot a\cdot \lambda) +c\cdot a, \label{eq:cost}
\end{align}
where the first part $x+1- x \cdot a\cdot \lambda$ is the resulting age in the next slot and the second part is the incurred cost for updating the user. 
\end{itemize}

A \textit{policy} $\mu=\{A(0), A(1), \cdots\}$ of the MDP $\Omega$  specifies an action $A(t)$ for each slot $t$.  The \textit{average cost} $J(\pi)$ under policy $\mu$ is defined by
\begin{align*}
J(\mu)=\limsup_{T \rightarrow \infty} \frac{1}{T+1} E_{\mu} \left[ \sum_{t=0}^T C(\mathbf{S}(t), A(t))\right]. 
\end{align*}

Again, the objective of the MDP $\Omega$ is to find an $\Omega$-optimal policy defined as follows. 
\begin{define}
A policy $\mu$  of the MDP $\Omega$ is \textit{$\Omega$-optimal} if it minimizes the  average cost $J(\mu)$.
\end{define}
Traditionally, the Whittle index might be obtained by solving the optimality equation of $J(\mu)$, e.g. \cite{gittins2011multi,index:igor}. However, as  discussed, the average cost optimality equation  for the MDP $\Omega$ might not exit, and even if it exists, solving an optimality equation might be tedious. To look for a simpler way for obtaining the Whittle index, we  investigate structures of an $\Omega$-optimal policy instead, by looking at its discounted case  again. It turns out that our structural results will further simplify the derivation of the Whittle index. 

\subsection{Characterization of the $\Omega$-optimality}
First, we show that an $\Omega$-optimal policy is stationary  deterministic as follows.

\begin{theorem} \label{theorem:stationary-omgea}
There exists a  deterministic stationary  policy that is $\Omega$-optimal, independent of the initial state.
\end{theorem}
\begin{proof}
Please see Appendix \ref{appendix:theorem:stationary-omgea}. 
\end{proof}

Next, we show that an $\Omega$-optimal policy is a special type of  deterministic stationary  policies.

\begin{define}
A \textit{threshold-type} policy is a special deterministic stationary   policy of the MDP $\Omega$. The action for state $(x,0)$ is to idle, for all $x$. 
Moreover, if the  action for state $(x,1)$ is to update, then the action for  state $(x+1,1)$ is to update as well. In other words, there exists a threshold $\bar{X} \in \{1,2, \cdots\}$ such that the action is to update if there is an arrival and the age is greater than or equal to $\bar{X}$; otherwise, the action is  to idle. 
\end{define}

\begin{theorem} \label{theorem:threshold}
If the update cost $c\geq 0$, then there exists a threshold-type policy that is $\Omega$-optimal. 
\end{theorem}
\begin{proof}
It is obvious that an optimal action for state $(x,0)$ is to idle if $c \geq 0$. To establish the optimality of the threshold structure for state $(x,1)$, we need  the  \textit{discounted cost  optimality equation} for $J_{\alpha}(\mathbf{s})$, similar to Proposition~\ref{lemma:optimality-eq}: 
\begin{align*}
J_{\alpha}(\mathbf{s})=\min_{a \in \{0,1\}} C(\mathbf{s}, a)+\alpha E[J_{\alpha}(\mathbf{s}')].
\end{align*}

Similar to the proof of Theorem~\ref{theorem:optimal-switch}, we can focus on the discounted cost case and show that an $\Omega_{\alpha}$-optimal policy is the threshold type. Let $\mathscr{J}_{\alpha}(\mathbf{s};a)=C(\mathbf{s}; a)+\alpha E[J_{\alpha}(\mathbf{s}')]$. Then, $J_{\alpha}(\mathbf{s})=\min_{a \in \{0,1\}} \mathscr{J}_{\alpha}(\mathbf{s};a)$. Moreover, an $\Omega_{\alpha}$-optimal action for state $\mathbf{s}$ is $\argmin_{a \in \{0,1\}} \mathscr{J}_{\alpha}(\mathbf{s};a)$. Suppose that an $\Omega_{\alpha}$-optimal action for state $(x,1)$ is to update, i.e., 
\begin{align*}
\mathscr{J}_{\alpha}(x,1;1)-\mathscr{J}_{\alpha}(x,1;0) \leq 0.
\end{align*}
Then, an $\Omega_{\alpha}$-optimal action for state $(x+1,1)$ is still to update since
\begin{align*}
&\mathscr{J}_{\alpha}(x+1,1;1)-\mathscr{J}_{\alpha}(x+1,1;0)\\
=&\left(1+c+\alpha E[J_{\alpha}(1,\lambda')] \right) -\left(x+2+\alpha E[J_{\alpha}(x+2,\lambda')] \right)\\
\mathop{\leq}^{(a)} & \left(1+c+\alpha E[J_{\alpha}(1,\lambda')] \right) -\left(x+1+\alpha E[J_{\alpha}(x+1,\lambda')] \right)\\
=&J_{\alpha}(x,1;1)-J_{\alpha}(x,1;0) \leq 0,
\end{align*}
where  (a) results from the non-decreasing function of $J_{\alpha}(x,\lambda)$ in $x$ given $\lambda$ (similar to Proposition~\ref{lemma:monotone}). Hence, an $\Omega_{\alpha}$-optimal policy is  threshold-type.  
\end{proof}

Thus far, we have successfully identify the threshold structure of an $\Omega$-optimal policy. The MDP $\Omega$  then can be reduced to a two-dimensional discrete-time Markov chain (DTMC) by  applying a threshold-type policy. To find an optimal threshold for minimizing the average cost,  in the next lemma we explicitly derive the average cost for a threshold-type policy. 
\begin{figure}[!t]
\centering
\includegraphics[width=.35\textwidth]{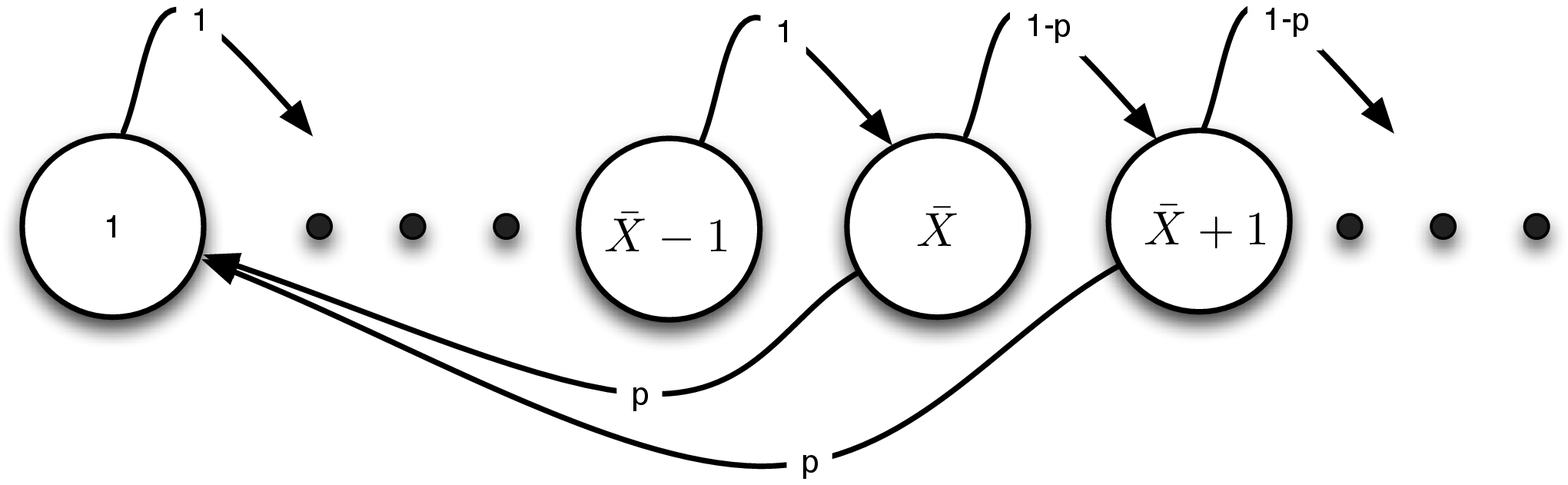}
\caption{The post-action age $\tilde{X}(t)$ under the threshold-type policy forms a DTMC.}
\label{fig:post-age}
\end{figure}
\begin{lemma} \label{lemma:threshold-cost}
Given the threshold-type policy $\mu$ with the threshold $\bar{X}\in \{1, 2, \cdots\}$, then 
 the average  cost $J(\mu)$, denoted by $\mathscr{C}(\bar{X})$, under the policy is 
\begin{align}
\mathscr{C}(\bar{X})=\frac{\frac{\bar{X}^2}{2}+(\frac{1}{p}-\frac{1}{2})\bar{X}+\frac{1}{p^2}-\frac{1}{p}+c}{\bar{X}+\frac{1-p}{p}}. \label{eq:threshold-cost}
\end{align}
\end{lemma}
\begin{proof}
Let $\tilde{X}(t)$ be the  age \textit{after} an action in slot $t$; precisely, if $\mathbf{S}(t)=(x,\lambda)$ and $A(t)=a$, then $\tilde{X}(t)=x+1-x\cdot a\cdot \lambda$. Note that $\tilde{X}(t)$, called  \textit{post-action age} (similar to the post-decsion state \cite{AMDP:Powell,online-engery:borkar}), is  different from the \textit{pre-action age} $X(t)$.  Then, the post-action age $\tilde{X}(t)$ by the threshold-type policy forms an one-dimensional DTMC in Fig. \ref{fig:post-age}, with the transition probabilities being
\begin{align*}
&P[\tilde{X}(t+1)=i+1|\tilde{X}(t)=i]=1, \text{\,\,for\,\,}i=1, \cdots, \bar{X}-1;  \\
&P[\tilde{X}(t+1)=i+1|\tilde{X}(t)=i]=1-p, \text{\,\,for\,\,}i=\bar{X}, \cdots;\\  
&P[\tilde{X}(t+1)=1|\tilde{X}(t)=i]=p, \text{\,\,for\,\,}i=\bar{X}, \cdots.  
\end{align*}

To calculate the average cost of the policy, we  associate each state in the DTMC with a cost. The DTMC incurs the cost of $c+1$ in slot $t$ when the post-action age in slot $t$ is $\tilde{X}(t)=1$. That is because the post-action age $\tilde{X}(t)=1$ implies that the BS updates the user. In addition, the DTMC incurs the age cost of $y$ in slot $t$ when the post-action age is $\tilde{X}(t)=y \neq 1$. 

The steady-state distribution $\boldsymbol{\xi}=(\xi_1, \xi_2, \cdots)$ of the DTMC can be solved as
\begin{align*}
\xi_i=\left\{
\begin{array}{ll}
\frac{1}{\bar{X}+\frac{1-p}{p}} & \text{if $i=1, \cdots, \bar{X}$;}\\
\frac{1}{\bar{X}+\frac{1-p}{p}}(1-p)^{i-\bar{X}} & \text{if $i=\bar{X}+1, \cdots$}.
\end{array}
\right.
\end{align*}
Therefore, the average cost of the DTMC is 
\begin{align*}
(1+c)\xi_1 +\sum_{i=2}^{\infty} i\xi_i=\frac{\frac{\bar{X}^2}{2}+(\frac{1}{p}-\frac{1}{2})\bar{X}+\frac{1}{p^2}-\frac{1}{p}+c}{\bar{X}+\frac{1-p}{p}}.
\end{align*}
\end{proof}
We remark that the post-action age  introduced in the above proof are beneficial in many aspects:
\begin{itemize}
	\item The post-action age can form an one-dimensional DTMC, instead of the original two-dimensional state $\mathbf{S}(t)$. 
	\item We cannot associate each pre-action age with a fixed cost, since the cost in Eq. (\ref{eq:cost}) depends on not only state but also action. Instead, the cost for each post-action age is determined by its age only. 
	\item The post-action age will  facilitate the online algorithm design in Section~\ref{section:online}.
\end{itemize}


%
%
\subsection{Derivation of the Whittle index}
Now, we are ready to define and derive the Whittle index as follows.  
\begin{define}
We define the Whittle index $I(\mathbf{s})$ by  the  cost $c$ that makes  both actions, to update and to idle, for state  $\mathbf{s}$ equally desirable. 
\end{define}

In the next theorem, we will obtain a very simple expression for the Whittle index  by combining Theorem~\ref{theorem:threshold} and Lemma~\ref{lemma:threshold-cost}. 
\begin{theorem} \label{theorem:whittle}
The Whittle index of the  sub-problem for state $(x,\lambda)$ is 
\begin{align}
I(x,\lambda)=\left\{
\begin{array}{ll}
0 & \text{if $\lambda=0$;}\\
\frac{x^2}{2}-\frac{x}{2}+\frac{x}{p} & \text{if $\lambda=1$.}
\end{array}
\right.\label{eq:index}
\end{align}
\end{theorem}
\begin{proof}
It is obvious that the Whittle index for state $(x,0)$ is $I(x,0)=0$ as  both actions result in the same immediate cost and the same age of next slot if $c=0$. 

Let $g(x)=\mathscr{C}(x)$ in Eq. (\ref{eq:threshold-cost})  for the domain of $\{x\in \mathbb{R}: x\geq 1\}$. Note that $g(x)$ is strictly convex in the domain. Let $x^*$ be the minimizer of $g(x)$. Then, an optimal threshold for minimizing the average cost $\mathscr{C}(\bar{X})$ is either $\lfloor x^* \rfloor$ or $\lceil x^* \rceil$: the optimal threshold is $\bar{X}^*=\lfloor x^* \rfloor$ if $\mathscr{C}(\lfloor x^* \rfloor) <  \mathscr{C}(\lceil x^* \rceil)$ and $\bar{X}^*=\lceil x^* \rceil$  if $\mathscr{C}(\lceil x^* \rceil) < \mathscr{C}(\lfloor x^* \rfloor)$. If there is a tie, both choices are optimal, i.e, equally desirable. 

Hence,  both actions for state $(x,1)$ are equally desirable if and only if the age $x$ satisfies
\begin{align}
\mathscr{C}(x)=\mathscr{C}(x+1), \label{eq:equal-desirable}
\end{align}
i.e., $x= \lfloor x^* \rfloor$ and both thresholds of $x$ and $x+1$ are optimal. 
By solving Eq.~(\ref{eq:equal-desirable}), we  obtain the cost, as stated in the theorem, to make both actions equally desirable. 
\end{proof}

According to Theorem \ref{theorem:whittle}, both actions might have a tie. If there is a tie, we break the tie in favor of idling. Then, we can explicitly express the optimal threshold  in the next theorem. 

\begin{lemma} \label{lemma:optimal-threshold}
The optimal threshold for minimizing the average cost $\mathscr{C}(\bar{X})$ is $x$ if the cost $c$ satisfies $I(x-1,1) \leq  c < I(x,1)$, for all $x=1, 2, \cdots$.
\end{lemma}
\begin{proof}
Please see Appendix~\ref{appendix:lemma:optimal-threshold}.
\end{proof}

Next, according to \cite{whittle}, we have to demonstrate the \textit{indexability}  such that the Whittle index is feasible.  
\begin{define}
For a given cost $c$, let $\mathbf{S}(c)$ be the set of states such that the optimal actions for the states are to idle. The sub-problem is \textit{indexable} if  the set $\mathbf{S}(c)$ monotonically increases from the empty set to the entire state space, as $c$ increases from $-\infty$ to $\infty$. 
\end{define}

\begin{theorem}
The sub-problem is indexable. 
\end{theorem}
\begin{proof}
If $c<0$, the optimal action for every state is to update; as such, $\mathbf{S}(0)=\emptyset$. If $c\geq 0$, then $\mathbf{S}(c)$ is composed of the set $\{\mathbf{s}=(x,0): x=1, 2, \cdots\}$  and a set of $(x,1)$ for some $x$'s. According to Lemma \ref{lemma:optimal-threshold}, the optimal threshold monotonically increases to infinity as $c$ increases, and hence the set $\mathbf{S}(c)$ monotonically increases to the entire state space. 
\end{proof}

\subsection{Index scheduling algorithm} \label{subsection:index}

Now, we are ready to propose a low-complexity \textit{index scheduling algorithm} based on the Whittle index.   For each slot $t$, the BS observes  age $X_i(t)$ and arrival indicator $\Lambda_i(t)$ for every user $u_i$; then, updates user $u_i$ with the highest value of the Whittle index $I(X_i(t),\Lambda_i(t))$, i.e., $D(t)=\argmax_{i=1, \cdots, N} I(X_i(t),\Lambda_i(t))$. We can think of the index $I(X_i(t),\Lambda_i(t))$ as a \textit{value} of updating  user $u_i$. The intuition of the index scheduling algorithm is that the BS intends to send the most valuable packet. 

The optimality of the index scheduling algorithm for the relaxed version is known \cite{gittins2011multi}. Next,  we show that the proposed index scheduling algorithm is age-optimal for the original  problem (without relaxation), when the packet arrivals for all users are \textit{stochastically identical}. 

\begin{lemma} \label{lemma:index-optimal}
If the arrival rates of all information sources are the same, i.e., $p_i=p_j$ for all $i \neq j$, then the index scheduling algorithm is age-optimal.
\end{lemma}
\begin{proof}
Note that, for this case, the index scheduling algorithm send an arriving packet with the largest age of information, i.e., $D(t)=\argmax_i X_i(t)\Lambda_i(t)$ for each slot $t$. Then, in Appendix~\ref{appendix:lemma:index-optimal} we show that the policy is $\Delta$-optimal. 
\end{proof}

In Section~\ref{section:simulation} we will further validate the index scheduling algorithm  for stochastically non-identical arrivals by simulations.


\section{Online scheduling algorithm design} \label{section:online}
Thus far, we have developed two scheduling algorithms in Sections~\ref{section:mdp} and \ref{section:whittle}. Both algorithms are  offline, as  the structural MDP scheduling algorithm and the index scheduling algorithm need the arrival statistics as prior information to pre-compute an optimal action for each virtual state and the Whittle index, respectively. To solve the more challenging case when the arrival statistics are unavailable, in this section we  develop online versions for both offline  algorithms.  

\subsection{An MDP-based online scheduling algorithm} \label{subsection:mdp-online}

We first develop an online version of the MDP scheduling algorithm by leveraging \textit{stochastic approximation} techniques \cite{learning-book:borkar}. The intuition is that,  instead of updating $V^{(m)}(\mathbf{s})$ for all virtual states in each iteration of Eq. (\ref{eq:rvia}), we update $V^{(m)}(\mathbf{s})$ by following a \textit{sample path}, which is a set of outcomes of the arrivals over slots. It turns out that the sample-path updates will converge to the $\Delta$-optimal solution. To that end, we need a \textit{stochastic version} of the RVIA. However, the RVIA in Eq. (\ref{eq:rvia}) is not suitable  because the expectation  is inside the minimization (see \cite{AMDP:Powell} for details). While minimizing the RHS of Eq. (\ref{eq:rvia}) for a given current state, we would need the transition probabilities to calculate the expectation. To tackle this, we design \textit{post-action states} for our problem, similar to the proof of Lemma~\ref{lemma:threshold-cost}.

We define  post-action state $\tilde{\mathbf{s}}$ as the ages and the arrivals \textit{after} an action. The state we used before is referred to as the \textit{pre-action} state. If $\mathbf{s}=(x_1, \cdots, x_N, \lambda_1, \cdots, \lambda_N) \in \mathbf{S}^{(m)}$ is a  virtual  state of the MDP $\Delta^{(m)}$, then the  virtual  post-action state after action $d$ is $\tilde{\mathbf{s}}=(\tilde{x}_1, \cdots, \tilde{x}_N, \tilde{\lambda}_1, \cdots, \tilde{\lambda_N})$ with  
\begin{align*}
\tilde{x}_i=\left\{
\begin{array}{ll}
1 & \text{if $i=d$ and $\lambda_i=1$;} \\
\left[x_i+1\right]^+_m & \text{else},
\end{array}
\right.
\end{align*}
and $\tilde{\lambda}_i=\lambda_i$ for all $i$. 

Let $\tilde{V}^{(m)}(\tilde{\mathbf{s}})$ be the value function based on the post-action states defined by
\begin{align*}
\tilde{V}^{(m)}(\tilde{\mathbf{s}})=E[V^{(m)}(\mathbf{s})],
\end{align*}
where the expectation  is taken over all possible the pre-action states $\mathbf{s}$ reachable from the post-action state.
We can then  write down the post-action average cost optimality equation \cite{AMDP:Powell} for the  virtual  post-action  state $\tilde{\mathbf{s}}=(\tilde{x}_1, \cdots, \tilde{x}_N,\tilde{\lambda}_1, \cdots, \tilde{\lambda}_N)$:
\begin{align*}
&\tilde{V}^{(m)}(\tilde{\mathbf{s}})+V^{(m)*}\\
=&E\left[\min_{d \in \{0, 1, \cdots, N\}} C\left((\tilde{\mathbf{x}},\tilde{\boldsymbol{\lambda}}'),d\right)\right.\\
&\left.\hspace{1cm} +\tilde{V}^{(m)}([\tilde{\mathbf{x}}+\mathbf{1}-\tilde{\mathbf{x}}_d \tilde{\boldsymbol{\lambda}}'_d]^+_m,\tilde{\boldsymbol{\lambda}}')\right],
\end{align*}
where $\tilde{\boldsymbol{\lambda}'}$ is the   next  arrival vector;  $\tilde{\mathbf{x}}_i=(0, \cdots, \tilde{x}_i, \cdots,0)$ denotes the zero vector except for the $i$-th entry being replaced by $\tilde{x}_i$; $\tilde{\boldsymbol{\lambda}}'_i=(0, \cdots, \tilde{\lambda}_i, \cdots,0)$ denotes the zero vector except for the $i$-th entry being replaced by $\tilde{\lambda}_i$; the vector $\mathbf{1}=(1, \cdots, 1)$ is the unit vector. 
From the above optimality equation, the RVIA  is as follows: 
\begin{align}
\tilde{V}^{(m)}_{n+1}(\tilde{\mathbf{s}})=&E\left[\min_{d \in \{0, 1, \cdots, N\}} C\left((\tilde{\mathbf{x}},\tilde{\boldsymbol{\lambda}}'),d\right)\right.\nonumber\\
&\left.+\tilde{V}^{(m)}_n([\tilde{\mathbf{x}}+\mathbf{1}-\mathbf{x}_d  \tilde{\boldsymbol{\lambda}}'_d]^+_m,\tilde{\boldsymbol{\lambda}}')\right] - \tilde{V}^{(m)}_n(\mathbf{0}).  \label{eq:rvia-post}
\end{align}

Subsequently, we propose the \textit{MDP-based online  scheduling algorithm} in Alg.~\ref{alg:online} based on the stochastic version of the RVIA. In Lines~\ref{alg:online-init-start}-\ref{alg:online-init-end}, we initialize $\tilde{V}^{(m)}(\tilde{\mathbf{s}})$ of all  virtual  post-action states and start from the reference point. Moreover, by $v$ we record $\tilde{V}^{(m)}(\tilde{\mathbf{s}})$ of the current  virtual  post-action state.
By observing the current arrivals $\mathbf{\Lambda}(t)$ and plugging in Eq.~(\ref{eq:rvia-post}),  the expectation in Eq. (\ref{eq:rvia-post}) can be removed; as such, in Line \ref{alg:online-optimal-decision}  we  optimally update a user by minimizing Eq.~(\ref{eq:optimal-online-decision}). Then, we update $\tilde{V}^{(m)}(\tilde{\mathbf{s}})$ of the current  virtual  post-action state in Line~\ref{alg:online-value-update}, where $\gamma(t)$ is a \textit{stochastic step-size} in slot~$t$ to strike a balance between the previous $\tilde{V}^{(m)}(\tilde{\mathbf{s}})$ and the updated value $v$. Finally, the next  virtual  post-action state is updated in Lines~\ref{alg:online-state-update} and \ref{alg:online-state-update2}

\begin{algorithm}[!t]
\SetCommentSty{text}
\SetAlgoLined 
\SetKwFunction{Union}{Union}\SetKwFunction{FindCompress}{FindCompress} \SetKwInOut{Input}{input}\SetKwInOut{Output}{output}

\tcc{\textit{\textbf{Initialization}}}
$\tilde{V}^{(m)}(\tilde{\mathbf{s}}) \leftarrow 0$ for all states $\tilde{\mathbf{s}}\in \mathbf{S}^{(m)}$;\\ \label{alg:online-init-start}
$\tilde{\mathbf{s}} \leftarrow \mathbf{0}$;\\ 
$v \leftarrow 0$;\\ \label{alg:online-init-end}

\While{$1$}{
\tcc{\textit{\textbf{Decision in slot $t$}}}
We optimally make a decision $D^*(t)$ in slot $t$ according to the current arrivals $\boldsymbol{\Lambda}(t)=(\Lambda_1(t), \cdots, \Lambda_N(t)$ in slot $t$: \label{alg:online-optimal-decision}
\begin{align}
D^*(t)=&\argmin_{d \in \{0, 1, \cdots, N\}}  C\left((\tilde{\mathbf{x}},\boldsymbol{\Lambda}(t)),d\right)\nonumber\\
&+\tilde{V}^{(m)}([\tilde{\mathbf{x}}+\mathbf{1}-\tilde{\mathbf{x}}_d \boldsymbol{\Lambda}_d(t)]^+_m,\boldsymbol{\Lambda}(t)); \label{eq:optimal-online-decision}
\end{align} 

\tcc{\textit{\textbf{Value update}}}
$v \leftarrow   C\left((\tilde{\mathbf{x}},\boldsymbol{\Lambda}(t)),D^*(t)\right)+\tilde{V}^{(m)}([\tilde{\mathbf{x}}+\mathbf{1}-\tilde{\mathbf{x}}_{D^*(t)} \boldsymbol{\Lambda}_{D^*(t)}(t)]^+_m ,\boldsymbol{\Lambda}(t))-\tilde{V}^{(m)}(\mathbf{0})$;\\

$\tilde{V}^{(m)}(\tilde{\mathbf{s}})\leftarrow (1-\gamma(t))\tilde{V}^{(m)}(\tilde{\mathbf{s}})+\gamma(t) v$; \\\label{alg:online-value-update}

\tcc{\textit{\textbf{post-action state update}}}
$\tilde{\mathbf{x}} \leftarrow [\tilde{\mathbf{x}}+\mathbf{1}-\tilde{\mathbf{x}}_{D^*(t)} \boldsymbol{\Lambda}_{D^*(t)}(t)]^+_m$;\\ \label{alg:online-state-update}

$\tilde{\boldsymbol{\lambda}} \leftarrow \boldsymbol{\Lambda}(t)$. \label{alg:online-state-update2}
}

\caption{MDP-based online scheduling algorithm}
\label{alg:online}
\end{algorithm}

Next, we show the optimality of the MDP-based online scheduling algorithm as slot $t$ approaches infinity.

\begin{theorem}
If $\sum_{t=0}^{\infty} \gamma(t) = \infty$ and $\sum_{t=0}^{\infty} \gamma^2(t) < \infty$, then Alg.~\ref{alg:online} converges to $\Delta^{(m)}$-optimum. 
\end{theorem}
\begin{proof}
According to \cite{learning:borkar1,learning:borkar2}, we only need to verify that the truncated MDP is unichain, which  has been completed in Appendix \ref{appendix:theorem:truncation}.
\end{proof}

In the above theorem, $\sum_{t=0}^{\infty} \gamma(t) =\infty$ implies that Alg.~\ref{alg:online} needs an infinite number of iterations to learn the $\Delta$-optimal solution, while the offline Alg. \ref{alg:offline}  converges to the optimal solution in a finite number of iterations. Moreover, $\sum_{t=0}^{\infty} \gamma^2(t) < \infty$ means that the \textit{noise} from measuring $\tilde{V}^{(m)}(\tilde{\mathbf{s}})$ can be controlled. 
Finally, we want to emphasize that the proposed Alg.~\ref{alg:online} is asymptotically $\Delta$-optimal, i.e., it converges to the $\Delta$-optimal solution when both the truncation $m$ and the slot $t$  go to infinity. In Section VI, we will also numerically investigate the algorithm over finite  slots.

\subsection{An index-based online scheduling algorithm} \label{subsection:index-online}
Next,  we note that the simple Whittle index $I(x, \lambda)$ in Eq.~(\ref{eq:index}) depends on its arrival probability only. Thus, if the arrival probability is unknown, for each slot $t$ we revise the index by 
\begin{align*}
I(x,\lambda,t)=\left\{
\begin{array}{ll}
0 & \text{if $\lambda=0$;}\\
\frac{x^2}{2}-\frac{x}{2}+\frac{x}{p(t)} & \text{if $\lambda=1$,}
\end{array}
\right.
\end{align*}
where $$p(t)=\frac{\sum_{\tau=0}^t \Lambda(\tau)}{t+1}=\frac{p(t-1)\cdot t+\Lambda(t)}{t+1}$$ is the running average arrival rate. Then, we  propose the  \textit{index-based online scheduling algorithm} as follows. For each slot $t$, the BS observes  age $X_i(t)$ and arrival indicator $\Lambda_i(t)$ for every user $u_i$; then, calculate $p_i(t)$ and update user $u_i$ with the highest value of the revised Whittle index $I(X_i(t),\Lambda_i(t),t)$.

\section{Simulation results} \label{section:simulation}
In this section we conduct extensive computer simulations for the proposed four scheduling algorithms. We  demonstrate the switch-type structure of Alg.~\ref{alg:offline} in Section \ref{subsection:sim-switch}. In Section~\ref{subsection:sim-study} we  compare the proposed  scheduling algorithms, especially to validate the performance of the online algorithms over  finite  slots.   Finally, we study the wireless broadcast network with buffers at the BS in Section~\ref{subsection:sim-buffer}.

\subsection{Switch-type structure of Alg. \ref{alg:offline}} \label{subsection:sim-switch}
\begin{figure}[!h]
\centering
\includegraphics[width=.52\textwidth]{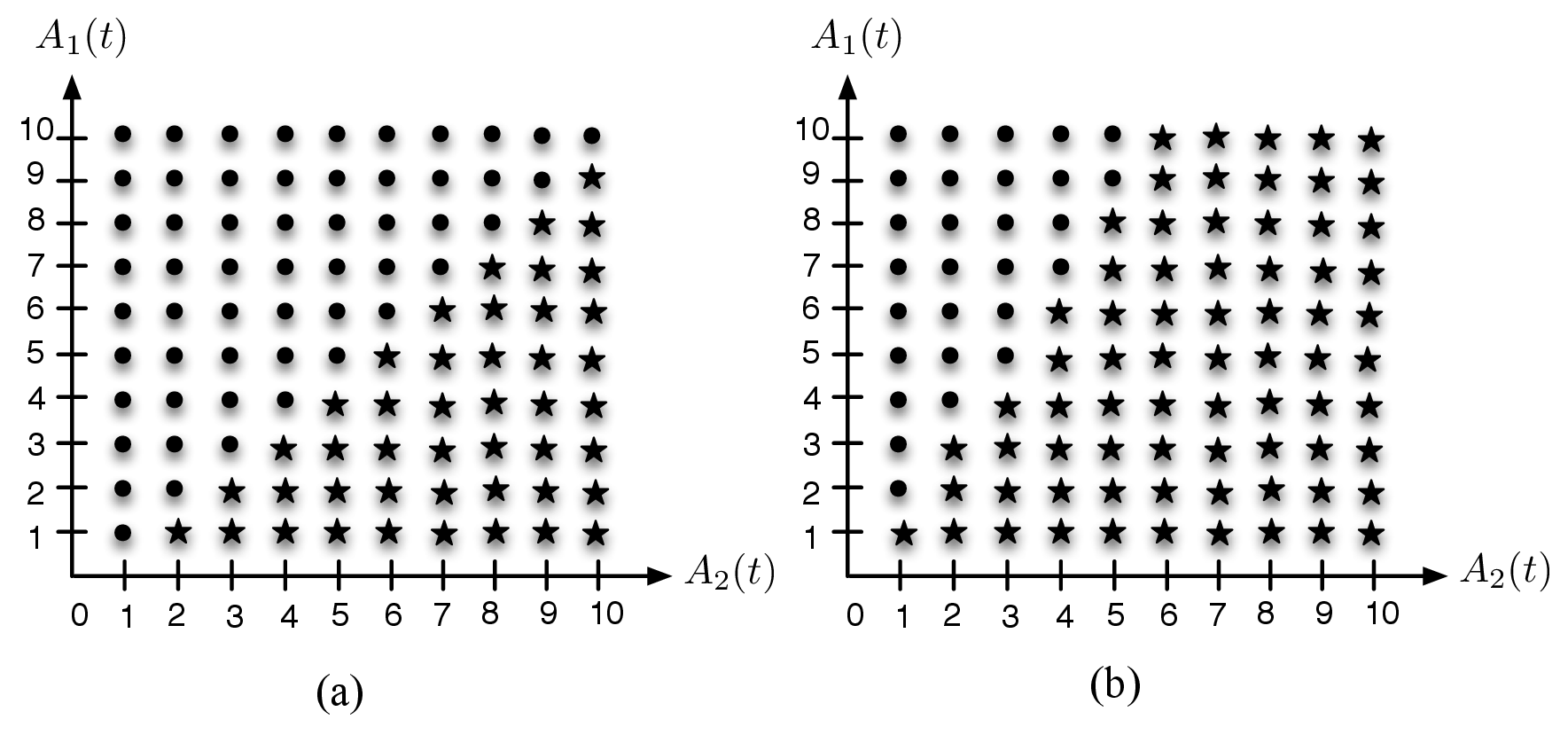}
\caption{Switch structure of Alg.~\ref{alg:offline} for (a) $p_1=p_2=0.9$; (b)  $p_1=0.9$, $p_2=0.5$. The dots  represent $D(t) = 1$ to update user $u_1$ and the stars  mean $D(t) = 2$ to update user $u_2$.}
\label{fig:switch}
\end{figure}

Figs. \ref{fig:switch}-(a) and \ref{fig:switch}-(b) show the switch-type structure of Alg.~\ref{alg:offline} for two users, when the BS has packets for both users. The experiment setting is as follows. We run Alg. \ref{alg:offline} with the boundary $m=10$ over  100,000 slots to search an optimal action for each virtual state. Moreover, we consider two arrival rate vectors, with $(p_1,p_2)$ being $(0.9, 0.9)$ and $(0.9, 0.5)$ in Figs. \ref{fig:switch}-(a) and \ref{fig:switch}-(b), respectively, where the \textit{dots} represent $D(t)=1$ and the \textit{stars} mean $D(t)=2$ when the BS has both arrivals in the same slot. We observe the switch structure in the figures, while Fig. \ref{fig:switch}-(a) is consistent with the index scheduling algorithm in Section~\ref{section:whittle} by simply comparing the ages of the two users.   Moreover,  by fixing the  arrival rate  $p_1=0.9$ for the first user, the BS will give a higher priority to the second user as $p_2$ decreases in Fig.~\ref{fig:switch}-(b). That is because the second user takes more time to wait for the next arrival and becomes a bottleneck.

\subsection{Numerical studies of the proposed scheduling algorithms} \label{subsection:sim-study}

\begin{figure}
\centering
\includegraphics[width=.45\textwidth]{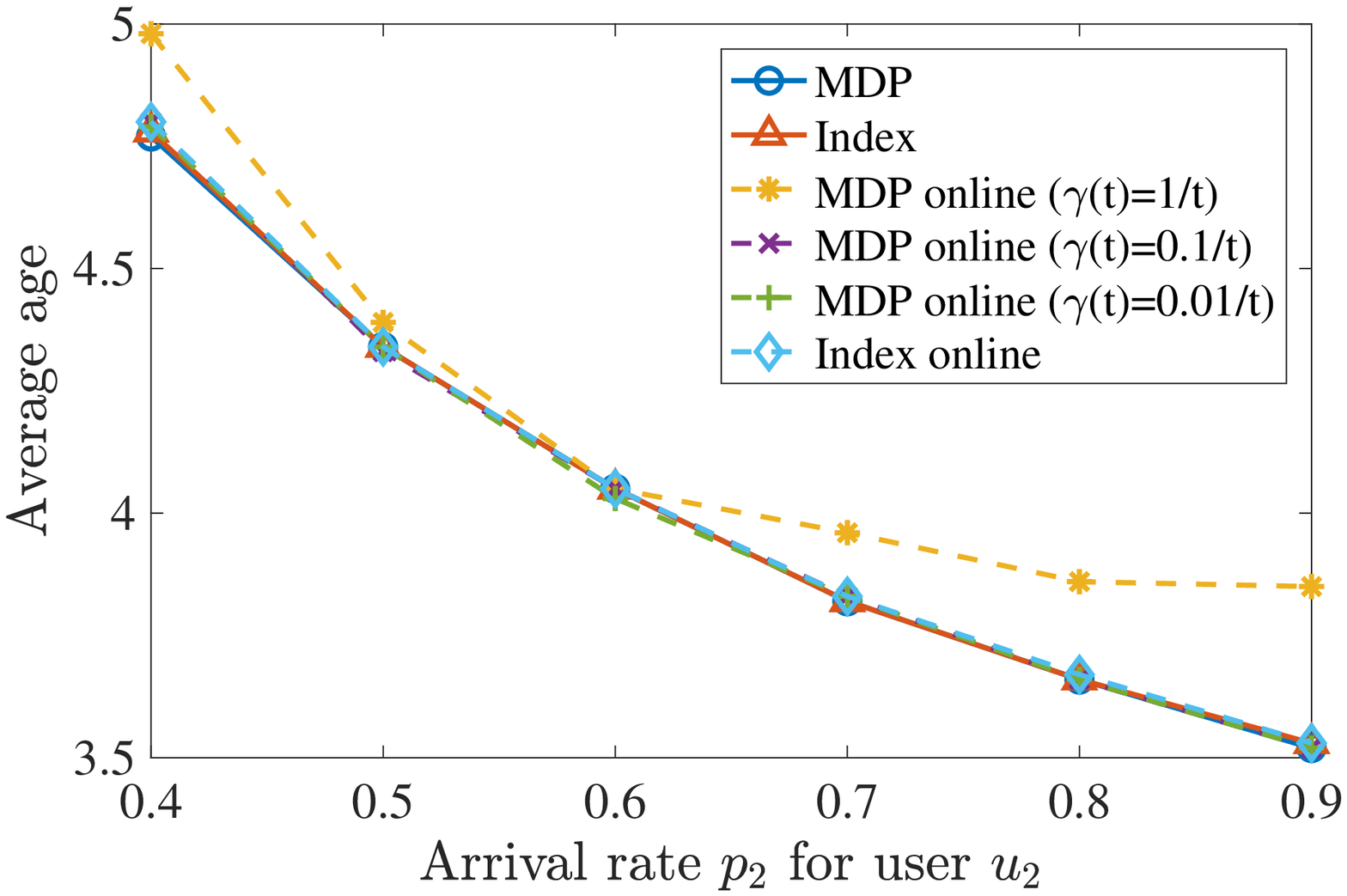}
\caption{Average age for different arrival rate $p_2$, where we fix $N=2$ and $p_1=0.6$.}
\label{fig:2user1}
\end{figure}
\begin{figure}
\centering
\includegraphics[width=.45\textwidth]{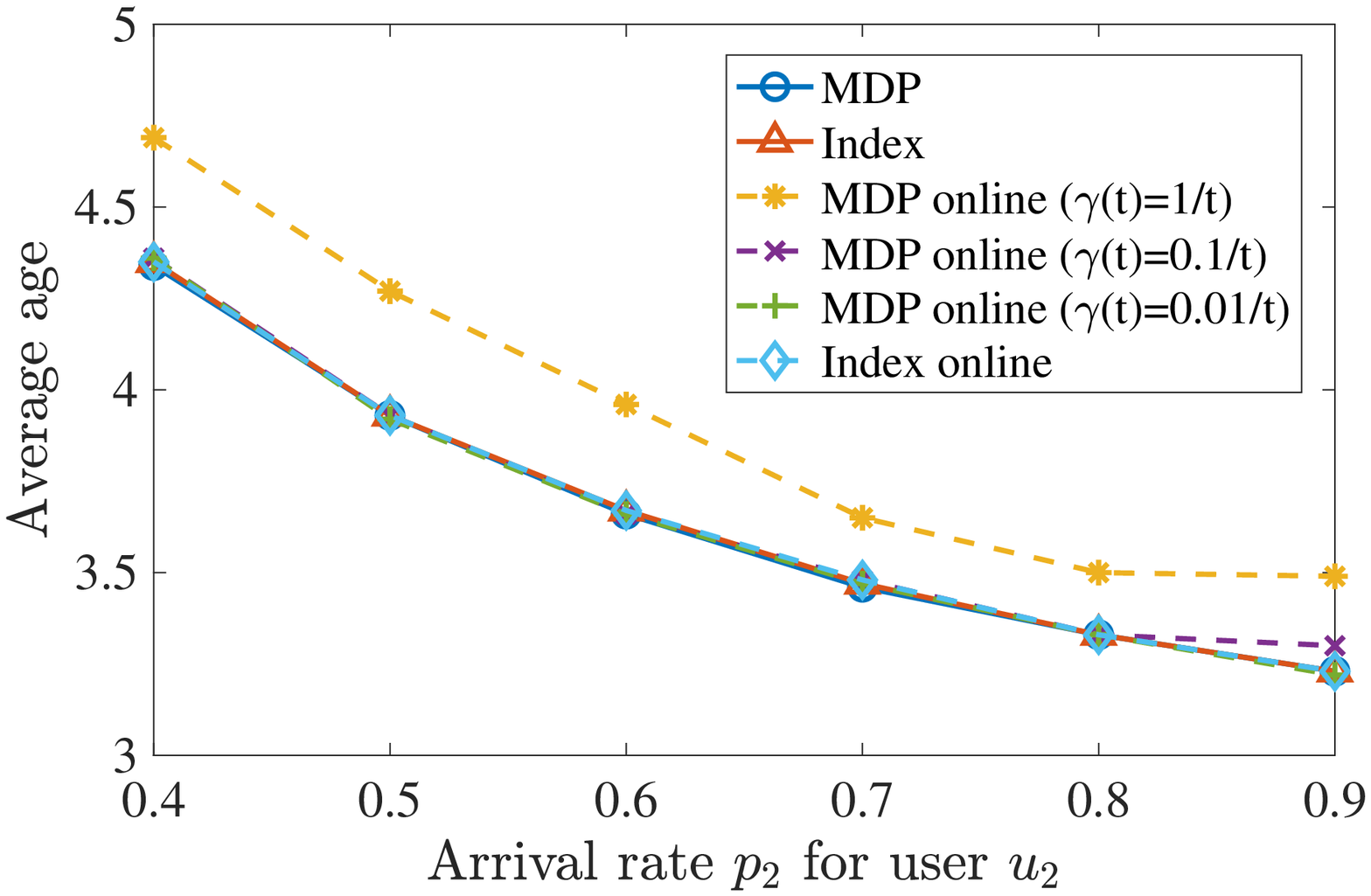}
\caption{Average age for different arrival rate $p_2$, where we fix $N=2$ and $p_1=0.8$.}
\label{fig:2user2}
\end{figure}
\begin{figure}
\centering
\includegraphics[width=.45\textwidth]{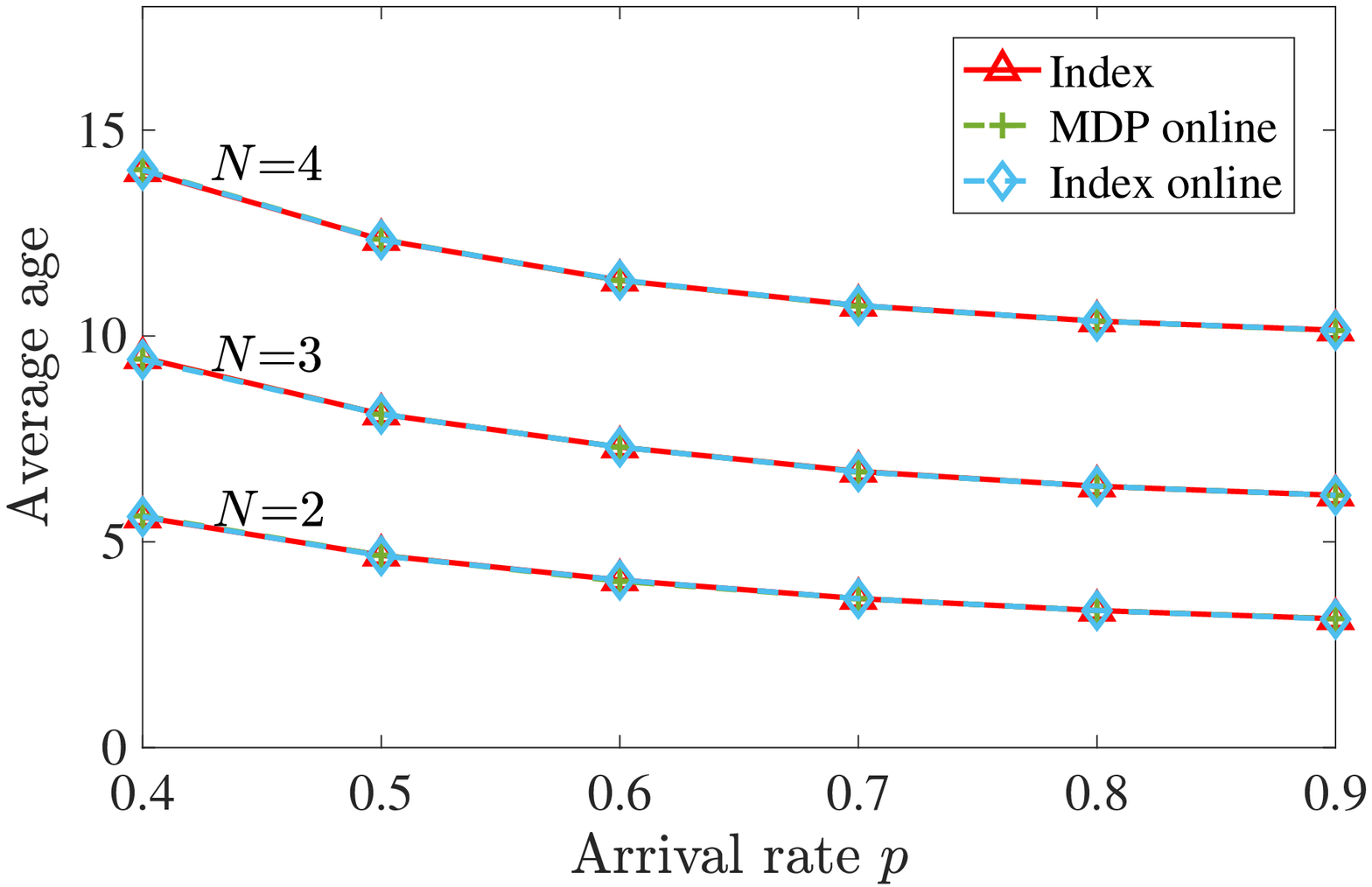}
\caption{Average age for different arrival rate $p_1=p_2=p$, where we fix $N=2,3,4$, respectively.}
\label{fig:4user}
\end{figure}
\begin{figure}
\centering
\includegraphics[width=.45\textwidth,trim={0 .1cm 0 .5cm},clip]{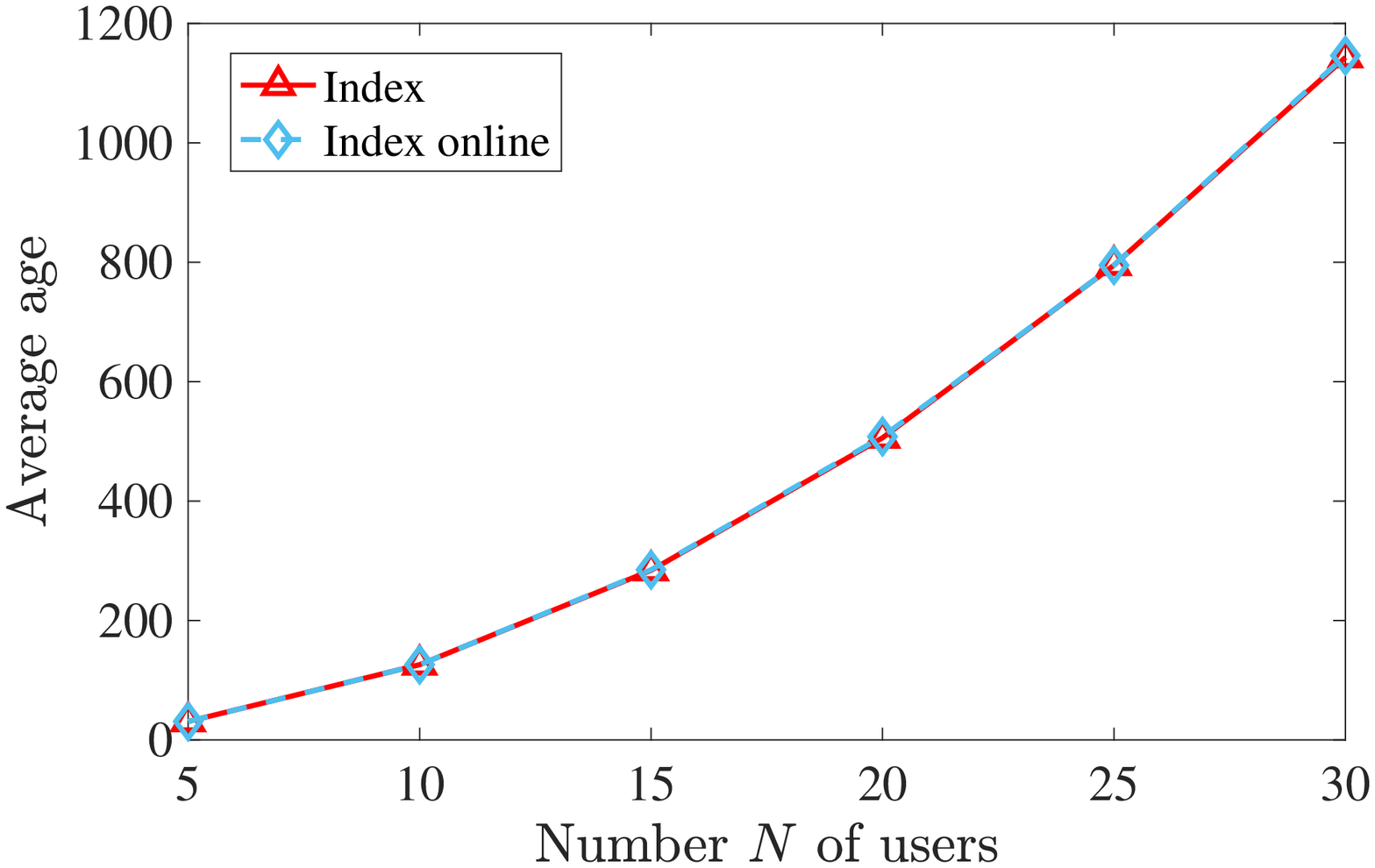}
\caption{Average age for different different number $N$ of users, where we fix the arrival rate $p_1=p_2=1/N$.}
\label{fig:many-user}
\end{figure}

In this section, we  examine  the proposed  four algorithms from various perspectives. First, we show the average age of two users for different $p_2$ in Figs.~\ref{fig:2user1} and \ref{fig:2user2} with fixed $p_1=0.6$ and $p_1=0.8$, respectively. Here, we set the boundary $m=30$ for the structural MDP scheduling algorithm. For the MDP-based online scheduling algorithms, we set the boundary $m=100$; moreover, we consider different step sizes in both figures, i.e., $\gamma(t)=1/t$, $\gamma(t)=0.1/t$, and $\gamma(t)=0.01/t$. All the results are averaged over 100,000 slots.  How to choose the best step size with provably performance guarantee is  interesting, but is out of  scope of this paper. By simulation, we observe that $\gamma(t)=0.01/t$ works perfectly for our problem to achieve the minimum average age. Moreover, comparing with the structural MDP scheduling algorithm,  the low-complexity index algorithm almost achieves the minimum average age, with  invisible performance loss. Even without the knowledge of the arrival statistics,  the MDP-based online scheduling algorithm with $\gamma(t)=0.01/t$ and the index-based online scheduling algorithm are both close to the minimum average age. 

Second, we show the average age of more than two users  with  $N=2,3,4$, respectively, in Fig.~\ref{fig:4user}, where we consider $p_1=p_2=p$. According to Lemma~\ref{lemma:index-optimal}, the index scheduling algorithm is age-optimal; thus, we find that both online  scheduling algorithms are almost age-optimal again, where we use $\gamma(t)=0.01/t$ only.  

Third, we show that the average age of many users in Fig.~\ref{fig:many-user}. In this setting,  the two MDP-based scheduling algorithms may be unfeasible because the resulting huge state space; thus,  we  consider the two index-based scheduling algorithms only. We find that the low-complexity index-based online scheduling algorithm again achieves the minimum average age. 

 By these numerical studies, we would suggest implementing the index-based online scheduling algorithm. It is not only simple to implement practically, but also has good performance.


\subsection{Networks with buffers} \label{subsection:sim-buffer}
Thus far, we consider the no-buffer network only. 
Finally, we study the buffers at the BS to store the latest information for each user. Similar to Section~\ref{section:mdp} we can find an age-optimal scheduling by an MDP.  However, we need to redefine the states of the MDP $\Delta$. In addition to the age $X_i(t)$ of information at  \textit{user} $u_i$, by $Y_i(t)$ we define the \textit{initial age of the information} at the \textit{buffer} for user $u_i$; precisely, 
\begin{align*}
Y_i(t)=\left\{
\begin{array}{ll}
0 & \text{if $\Lambda(t)=1$;}\\
Y_i(t-1)+1 & \text{else}.
\end{array}
\right.
\end{align*}	
Then, we redefine the state by $\mathbf{S}(t)=\{X_1(t), \cdots, X_N(t),$ $Y_1(t), \cdots, Y_N(t)\}$. Moreover, the immediate cost is redefined as  
\begin{align*}
C(\textbf{S}(t), D(t)=d) 
=&\sum_{i=1}^N (X_i(t)+1)-(X_d(t)-Y_d(t)),
\end{align*}
where we define $Y_0(t)=0$ for all $t$. Then, similar to Section~\ref{section:mdp}, we can show that
\begin{itemize}
	\item  there exists a deterministic stationary policy that is age-optimal;
	\item the similar sequence of approximate MDPs  converges;
	\item an age-optimal scheduling algorithm is switch-type: for every user $u_i$, if a $\Delta$-optimal action at state $\mathbf{s}=(x_i, \mathbf{x}_{-i}, \mathbf{y})$ is $d^*_{(x_i, \mathbf{x}_{-i}, \mathbf{y})}=i$, then  $d^*_{(x_i+1,\mathbf{x}_{-i},\mathbf{y})}= i$, where $\mathbf{y}=(y_1, \cdots, y_N)$ is the vector of all initial ages.
\end{itemize}
We  then modify Alg.~\ref{alg:offline} for the network with the buffers, as an age-optimal scheduling algorithm.  

To study the effect of the buffers, we consider the truncated MDP with the boundary $m=30$ and generate arrivals with $p_1=p_2=p$.  After averaging the age over 100,000 slots, we obtain the  average  age in Fig. \ref{fig:buffer}  for various  $p$, where the red curve with the triangle markers indicates the no-buffer networks (by employing Alg.~\ref{alg:offline}) and  the blue curve with the star markers indicates the network with the buffers.

In this setting, we see mild improvement of the average  age  by exploiting  the buffers. The buffers reduce the average age by only around $(5.6-5.3)/5.6\approx 5\%$ when $p=0.4$, and even lower when $p$ is higher. Let us discuss  the following three cases when \textit{both users have  arrivals in some slot}: 
\begin{itemize}
	\item \textit{When both $p_1$ and $p_2$ are high}: That means the user who is not updated currently  has a new arrival in the next slot with a high probability; as such, the old packet in the buffer seems not that effective.  
	\item   \textit{When both $p_1$ and $p_2$ are low}: Then,  the possibility of the two arrivals in the same slot is very low. Hence, this would be a trivial case. 
	\item \textit{When one of $p_1$ and $p_2$ are high and the other is low}: In this case, the BS will give the user with the lower arrival rate a higher update priority, as a packet for the other user will arrive shortly. 
\end{itemize}
According to the above discussions, we observe that the buffers might not be that effective as expected. The  no-buffer network is not only simple for practical implementation but also works  well.

\begin{figure}[!t]
\centering
\includegraphics[width=.45\textwidth]{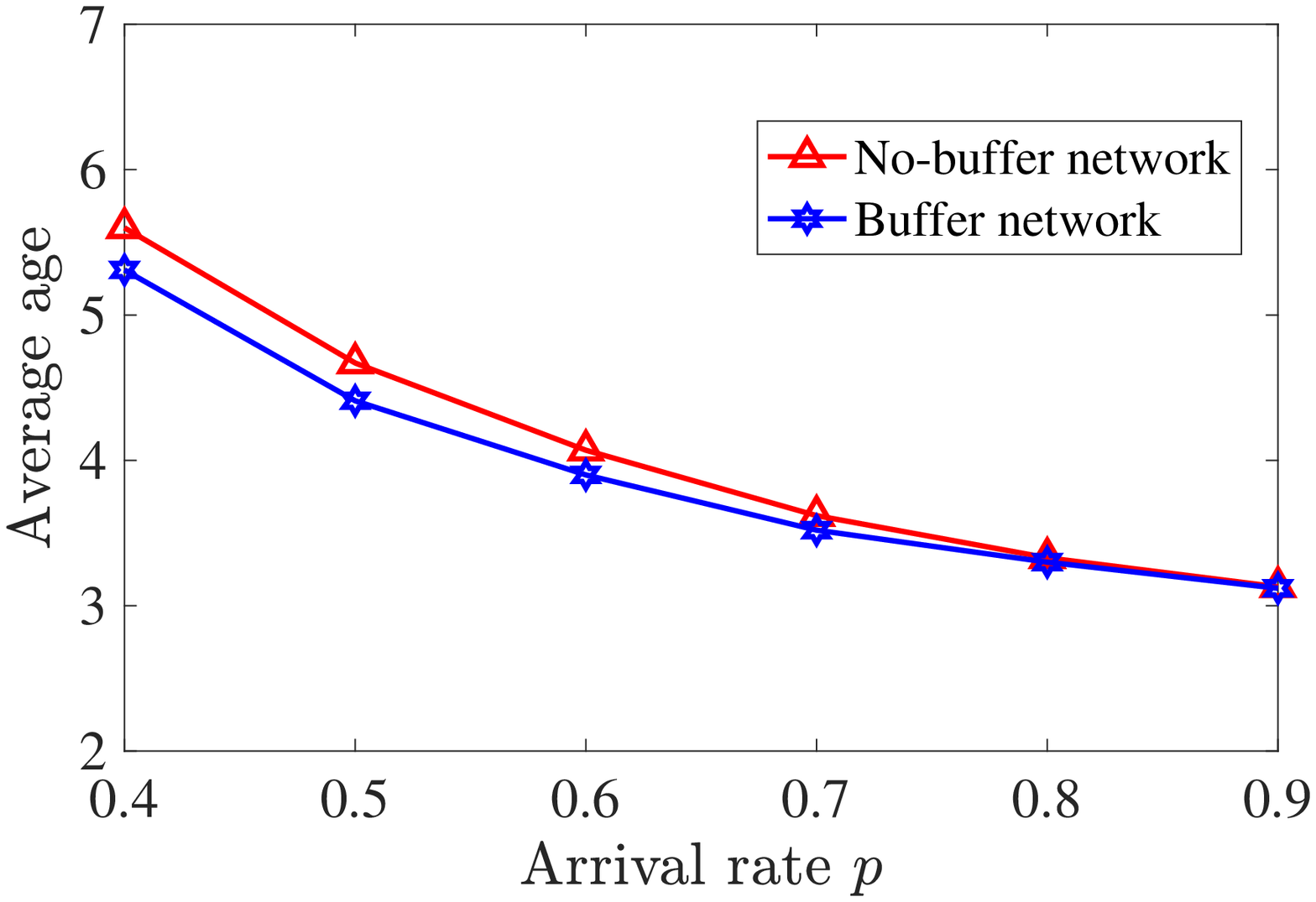}
\caption{Minimum average age for the network with/without the buffers by running the MDP-based scheduling algorithms, where  the arrival rates $p_1=p_2=p$.}
\label{fig:buffer}
\end{figure}

\section{Concluding remarks}
In this paper, we treated a wireless broadcast network, where many users are interested in different information that should be delivered by a base-station. We studied the age of information by designing and analyzing  four scheduling algorithms, i.e., the structural MDP scheduling algorithm, the index scheduling algorithm, the MPD-based online scheduling algorithm, and the index-based online scheduling algorithm. We not only theoretically investigated the optimality of the proposed algorithms, but also validate their performance via the computer simulations. It turns out that the low-complexity index scheduling algorithm and both online scheduling algorithms almost achieve the minimum average age. 

Some possible future works are discussed as follows. We focused on the no-buffer network in this paper.  It is an issue to study provable effectiveness of the buffers and to characterize the regime under which the no-buffer network works with marginal performance loss. Moreover, it is interesting to investigate structural results like ours for simplifying the calculation of the Whittle index for networks with buffers.  Finally, the paper treated a single-hop network only. It is interesting to extend our results to multi-hop networks. 

\section*{Acknowledgments}
The work of Yu-Pin Hsu is supported by Ministry of Science and
Technology, Taiwan (Project No. 107-2221-E-305-007-MY3).

{\small
	\bibliographystyle{IEEEtran}
	\bibliography{IEEEabrv,ref}
}

\appendices

\section{Proof of  Proposition \ref{lemma:optimality-eq} } \label{appendix:lemma:optimality-eq}
According to \cite{stationary-policy:Sennott}, it suffices to show that $V_{\alpha}(\mathbf{s})< \infty$ for every initial state $\mathbf{s}$ and discount factor $\alpha$. Let  $f$ be the deterministic stationary policy of  the MDP $\Delta$ that chooses $D(t)=0$ for all $t$. Note that, for initial state $\mathbf{s}=(x_1, \cdots, x_N, \lambda_1, \cdots, \lambda_N)$, we have
\begin{align*}
 V_{\alpha}(\mathbf{s};f)=& \lim_{T \rightarrow \infty}E_{f}\Bigl[\sum_{t=0}^{T} \alpha^t C(\mathbf{S}(t), D(t)) | \mathbf{S}(0)=\mathbf{s}\Bigr] \\
=&\sum_{t=0}^{\infty} \alpha^t \left[(x_1+t)+\cdots+(x_N+t)\right]\\
=&\frac{x_1+\cdots+x_N}{1-\alpha}+ \frac{\alpha N}{(1-\alpha)^2} < \infty.
\end{align*}
By  definition of the optimality,  we conclude  $V_{\alpha}(\mathbf{s}) < \infty$ since $V_{\alpha}(\mathbf{s}; f) < \infty$.

\section{Proof of Proposition \ref{lemma:monotone} } \label{appendix:lemma:monotone}
The proof is based on  induction on $n$ of the value iteration in  Eq~(\ref{eq:discount-itr}). 
The result clearly holds for  $V_{\alpha, 0}(\mathbf{s})$. Suppose that $V_{\alpha, n}(\mathbf{s})$ is non-decreasing in $x_i$. First,  note that the immediate cost $C(\mathbf{s},d)=\sum_{i=1}^N(x_i+1)-x_d \lambda_d$ is a non-decreasing function in $x_i$. Second,  $E[V_{\alpha,n}(\mathbf{s}')]$ is also a non-decreasing function in $x_i$ according to the induction hypothesis. 
 Since the minimum  operator (in Eq. (\ref{eq:discount-itr}))  holds the non-decreasing property, we conclude that $V_{\alpha, n+1}(\mathbf{s})$ is a non-decreasing function as well.

\section{Proof of Lemma \ref{thm:stationary} } \label{appendix:thm:stationary}
 We divide  state space $\mathbf{S}$ into two disjoint sets $\mathbf{S}_1$ and $\mathbf{S}_2$, where the age vectors in the set $\mathbf{S}_1$ belongs to  $\mathbf{X}$. Then,  set $\mathbf{S}_2$ is  a transient set for any scheduling algorithm that will update each user for at lease once. We consider two cases as follow.

First, we focus on the MDP $\Delta$ with the restricted state space $\mathbf{S}_1$, and show the lemma holds.  According to \cite{stationary-policy:Sennott}, we need to verify that the following two conditions are satisfied. 
\begin{enumerate}
	
	\item  \textit{There exists a  deterministic stationary policy $f$ of the MDP $\Delta$  such that the resulting discrete-time Markov chain (DTMC) by the policy is irreducible, aperiodic,  and the average cost $V(f)$ is finite}:  We consider the  deterministic stationary policy $f$ as the one that updates a user with an arrival and the largest age. It is obvious that the resulting DTMC is irreducible and aperiodic. Next, we transform the age  of information  into an \textit{age-queueing network} in Fig.~\ref{fig:age-queueing} consisting of $N$ age-queues $q_1, \cdots, q_N$, age-packet arrivals to each queue, and a server. 
In each slot an age-packet arrives at the system, since the age increases by one for each slot.	In each slot a channel associated with queue $q_i$ is ON with probability $p_i$. For each slot the server can serve a queue with an \textit{infinite} number of age-packets if its channel is ON. Then, the long-run average total age-queue size is the average cost. Note that the arrival rate is interior of the \textit{capacity region} \cite{neely:book} of the age-queueing network. Moreover, the policy $f$ is the \textit{maximum weight scheduling} algorithm \cite{neely:book} that is shown to be throughput-optimal; as such, the  average age-queue size is finite. Thus, the average cost $V(f)$ is finite as well. 

	\begin{figure}[!t]
	\centering
	\includegraphics[width=.35\textwidth]{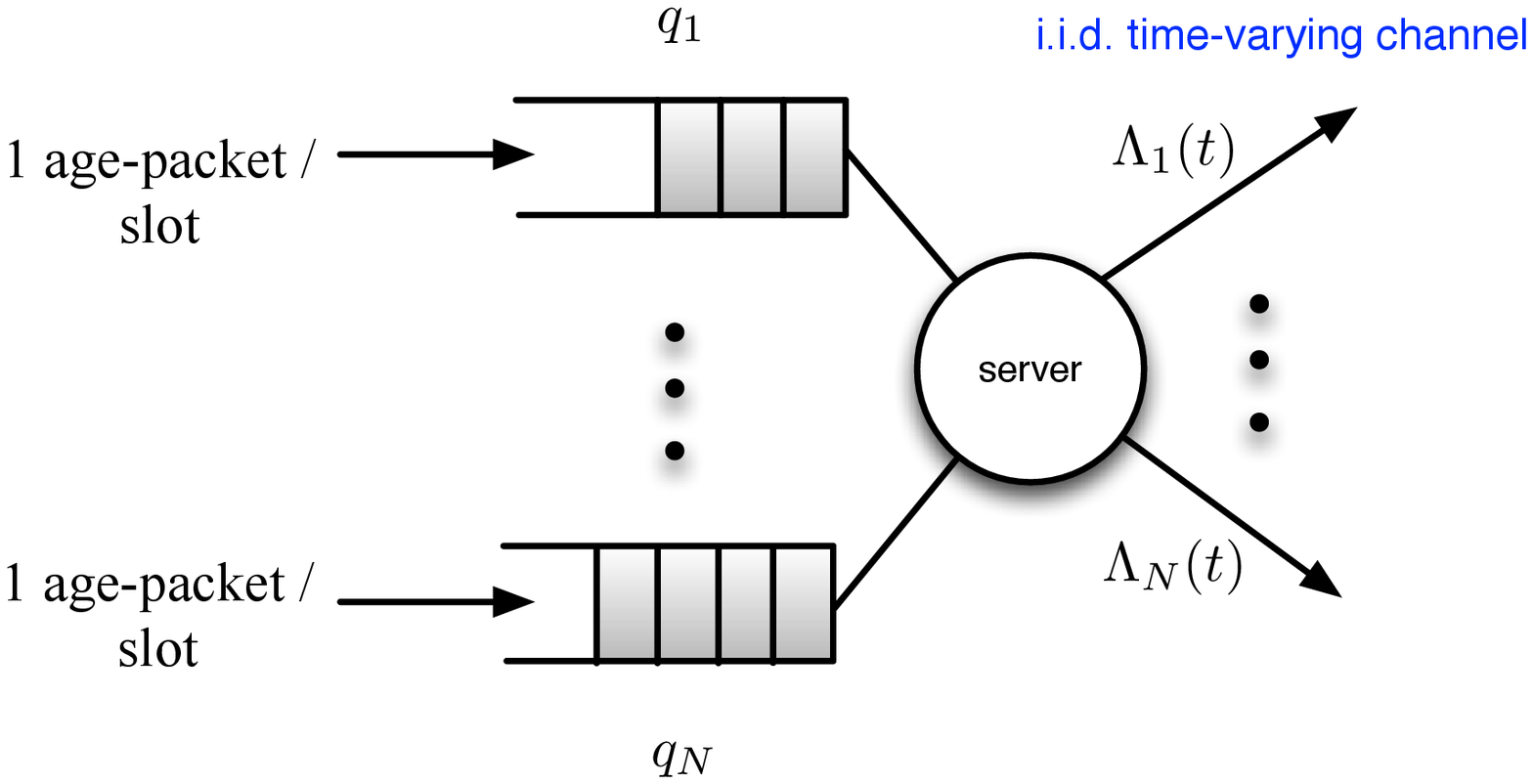}
	\caption{Age-queueing network.}
	\label{fig:age-queueing}
	\end{figure}
	
%
%
%
		\item \textit{There exists a nonnegative $L$ such that the relative cost function $h_{\alpha}(\mathbf{s}) \geq -L$ for all $\mathbf{s}$ and $\alpha$}: Let $C_{\mathbf{s},\mathbf{s}'}(\pi)$ be the expected cost of the first passage from state $\mathbf{s}$ to state $\mathbf{s}'$ under  policy $\pi$. Then, using the deterministic stationary policy $f$ in the first condition, we have $C_{\mathbf{s},\mathbf{s}'}(f) < \infty$ (see \cite[Proposition 4]{stationary-policy:Sennott}) and $ h_{\alpha}(\mathbf{s}) \geq - C_{\mathbf{0},\mathbf{s}}$ (see \cite[proof of Proposition 5]{stationary-policy:Sennott}). Moreover, as $V_{\alpha}(\mathbf{s})$ is a non-decreasing function in $x_i$ (see our Proposition \ref{lemma:monotone}), only  state  $\mathbf{s}$ with $x_i \leq N$ for all $i$ can probably result in a lower value of $V_{\alpha}(\mathbf{s})$ than $V_{\alpha}(\mathbf{0})$. We hence can choose $L=\max_{\mathbf{s}\in \mathbf{S}: x_i \leq N, \forall i} C_{\mathbf{0},\mathbf{s}}$.

\end{enumerate}
Thus,  according to \cite{stationary-policy:Sennott}, there exists a deterministic stationary policy that is $\Delta$-optimal and  minimum average cost is the constant $V^*$, independent of the initial state. 

Second, we note that, if the initial state  belongs to $\mathbf{S}_2$, then a $\Delta$-optimal policy will update each user for at least once (e.g, using the above deterministic scheduling algorithm $f$); otherwise, the average cost is infinite. In other words,  state $\mathbf{S}(t)$ will enter $\mathbf{S}_1$ in \textit{finite} time, and always stay in the set $\mathbf{S}_1$ onwards.   Thus, the average cost \textit{until the state enters $\mathbf{S}_1$} approaches zero as slots go to infinity, and the minimum average cost is still the constant $V^*$ as in the first case.  Moreover, there exists a deterministic stationary policy that is $\Delta$-optimal, e.g., following the deterministic stationary policy $f$ before entering $\mathbf{S}_1$ and then following the $\Delta$-optimal deterministic stationary policy in the first case.

\section{Proof of Theorem \ref{theorem:finite-approximation}} \label{appendix:theorem:finite-approximation}
Let $V^{(m)}_{\alpha}(\mathbf{s})$ and $h^{(m)}_{\alpha}(\mathbf{s})$ be the minimum  expected total $\alpha$-discounted cost and the relative cost function for the MDP $\Delta^{(m)}$, respectively. According to \cite{finite-state:Sennott}, we need to prove the following two conditions are satisfied. 
\begin{enumerate}
	\item \textit{There exists  a nonnegative $L$, a nonnegative  finite function $F(.)$ on $\mathbf{S}$ such that $-L \leq h^{(m)}_{\alpha}(\mathbf{s}) \leq F(\mathbf{s})$ for all  $\mathbf{s} \in \mathbf{S}^{(m)}$, where  $m=N+1, N+2, \cdots$ and $0<\alpha<1$}:  We consider a randomized stationary  algorithm $f$ that updates each user (with packet arrival) with equal probability for each slot. Similar to Appendix \ref{appendix:thm:stationary}, let $C_{\mathbf{s}, \mathbf{0}}(f)$ and $C^{(m)}_{\mathbf{s},\mathbf{0}}(f)$ be the expected cost from state $\mathbf{s} \in \mathbf{S}^{(m)}$ to the reference state $\mathbf{0}$ by applying the algorithm $f$ to $\Delta$ and $\Delta^{(m)}$, respectively. Then, $h^{(m)}_{\alpha}(\mathbf{s}) \leq C^{(m)}_{\mathbf{s},\mathbf{0}}(f)$ and  $C_{\mathbf{s},\mathbf{0}}(f) < \infty$ similar to Appendix~\ref{appendix:thm:stationary}.
	In the following, we will show that $C^{(m)}_{\mathbf{s},\mathbf{0}}(f) \leq C_{\mathbf{s},\mathbf{0}}(f)$ and then we can choose the function $F(\mathbf{s})=C_{\mathbf{s},\mathbf{0}}(f)$. 

To that end, we first express $P^{(m)}_{\mathbf{s}, \mathbf{s}'}(d)$ as
\begin{align*}
P^{(m)}_{\mathbf{s}, \mathbf{s}'}(d)=P_{\mathbf{s}, \mathbf{s}'}(d)+\sum_{\mathbf{r}(\mathbf{s}') \in \mathbf{S}-\mathbf{S}^{(m)}} P_{\mathbf{s}, \mathbf{r}(\mathbf{s}')}(d), 
\end{align*} 
for some (or no) \textit{excess probabilities}  \cite{finite-state:Sennott} on some state $\mathbf{r}(\mathbf{s}') \in \mathbf{S}-\mathbf{S}^{(m)}$,  depending on next state $\mathbf{s}'$. 
Since the scheduling algorithm $f$ is independent of the age,  given arrival vector $\boldsymbol{\lambda}$ we have $C_{(\mathbf{i},\boldsymbol{\lambda}),\mathbf{0}}(f) \leq C_{(\mathbf{j},\boldsymbol{\lambda}),\mathbf{0}}(f)$  for age vector $\mathbf{i} \leq \mathbf{j}$. Then, we obtain  
\begin{align}
&\sum_{\mathbf{s}' \in \mathbf{S}^{(m)}} P^{(m)}_{\mathbf{s},\mathbf{s}'}(d)C_{\mathbf{s}',\mathbf{0}}(f) \nonumber\\
=& \sum_{\mathbf{s}' \in \mathbf{S}^{(m)}} \bigl(P_{\mathbf{s},\mathbf{s}'}(d)+\sum_{\mathbf{r}(\mathbf{s}') \in \mathbf{S}-\mathbf{S}^{(m)}} P_{\mathbf{s},\mathbf{r}(\mathbf{s}')}(d)\bigr)C_{\mathbf{s}',\mathbf{0}}(f) \nonumber\\
\leq &\sum_{\mathbf{s}' \in \mathbf{S}^{(m)}} P_{\mathbf{s},\mathbf{s}'}(d)C_{\mathbf{s}',\mathbf{0}}(f)+\sum_{\mathbf{k} \in \mathbf{S}-\mathbf{S}^{(m)}} P_{\mathbf{s},\mathbf{k}}(d)C_{\mathbf{k},\mathbf{0}}(f) \nonumber\\
=&\sum_{\mathbf{s}' \in \mathbf{S}} P_{\mathbf{s},\mathbf{s}'}(d)C_{\mathbf{s}',\mathbf{0}}(f). \label{eq:c-inequ}
\end{align}

Using the above inequality, we then  conclude  $C^{(m)}_{\mathbf{s},\mathbf{0}}(f) \leq  C_{\mathbf{s},\mathbf{0}}(f) $ because
\begin{align*}
C^{(m)}_{\mathbf{s},\mathbf{0}}(f)=&E_{f}[C({\mathbf{s}},d)+\sum_{\mathbf{s}' \in \mathbf{S}^{(m)}} P^{(m)}_{\mathbf{s},\mathbf{s}'}(d)C_{\mathbf{s}',\mathbf{0}}(f)]\\
\leq & E_{f}[C(\mathbf{s},d)+\sum_{\mathbf{s}' \in \mathbf{S}} P_{\mathbf{s},\mathbf{s}'}(d) C_{\mathbf{s}',\mathbf{0}}(f)]\\
=&C_{\mathbf{s},\mathbf{0}}(f).
\end{align*}

On the other hand, we can choose $L=\max_{\mathbf{s} \in \mathbf{S}: x_i \leq N, \forall i} C_{\mathbf{0},\mathbf{s}}(f)$, since $h^{(m)}_{\alpha}(\mathbf{s}) \geq -C^{(m)}_{\mathbf{0}, \mathbf{s}}(f)$ (see Appendix~\ref{appendix:thm:stationary})  and $ -C^{(m)}_{\mathbf{0}, \mathbf{s}}(f)\geq -C_{\mathbf{0}, \mathbf{s}}(f)$	similar to above.

	\item \textit{The value $V^{(m)*}_{\infty}$ is  bounded by $V^*$, i.e., $ V^{(m)*}_{\infty}\leq V^*$}: 
We claim that $V^{(m)}_{\alpha}(\mathbf{s})\leq V_{\alpha}(\mathbf{s})$ for all $m$, and then the condition holds as
\begin{align*}
V^{(m)*}=&\limsup_{\alpha \rightarrow 1} (1-\alpha) V^{(m)}_{\alpha}(\mathbf{s}) \\
\leq& \limsup_{\alpha \rightarrow 1} (1-\alpha)V_{\alpha}(\mathbf{s})=V^*.  
\end{align*}

To verify this claim, we first note that  $V_{\alpha}(\mathbf{s})$ is a non-decreasing function in age (see Proposition~\ref{lemma:monotone}). Then, similar to Eq. (\ref{eq:c-inequ}), we have
\begin{align}
\sum_{\mathbf{s}' \in \mathbf{S}^{(m)}} P^{(m)}_{\mathbf{s},\mathbf{s}'}(d)V_{\alpha}(\mathbf{s}') 
\leq \sum_{\mathbf{s}' \in \mathbf{S}} P_{\mathbf{s},\mathbf{s}'}(d)V_{\alpha}(\mathbf{s}'). \label{eq:v-star-ineq}
\end{align}

We now prove the claim by induction on $n$ in Eq. (\ref{eq:discount-itr}). It is obvious when $n=0$. Suppose that $V^{(m)}_{\alpha,n}(\mathbf{s}) \leq V_{\alpha,n}(\mathbf{s})$, and then
\begin{align*}
&V^{(m)}_{\alpha,n+1}(\mathbf{s})\\
=&\min_{d \in \{0, 1, \cdots N\}} C(\mathbf{s},d)+ \alpha \sum_{\mathbf{s}' \in \mathbf{S}^{(m)}} P^{(m)}_{\mathbf{s},\mathbf{s}'}(d) V^{(m)}_{\alpha,n}(\mathbf{s}')\\
\mathop{\leq}^{(a)} & \min_{d \in \{0, 1, \cdots N\}} C(\mathbf{s},d)+ \alpha \sum_{\mathbf{s}' \in \mathbf{S}^{(m)}} P^{(m)}_{\mathbf{s},\mathbf{s}'}(d) V_{\alpha,n}(\mathbf{s}')\\
\mathop{\leq}^{(b)} & \min_{d \in \{0, 1, \cdots N\}} C(\mathbf{s},d)+ \alpha \sum_{\mathbf{s}' \in \mathbf{S}} P_{\mathbf{s},\mathbf{s}'}(d) V_{\alpha,n}(\mathbf{s}')\\
=&V_{\alpha,n+1}(\mathbf{s}),
\end{align*}	
where (a) results from the induction hypothesis, and (b) is due to Eq.~(\ref{eq:v-star-ineq}).
\end{enumerate}

\begin{remark} \label{remark to truncation}
Here, we want to emphasize that we have chosen $m > N$. If not, the state space $\mathbf{S}$ of the MDP $\Delta$ would have to include more \textit{transient} states such that all ages are no more than $N$, e.g., the age vector of $(N, \cdots, N)$. These additional states are not reachable from state $\mathbf{0}$. Thus, we cannot choose the $L$ as in the proof,  since $C_{\mathbf{0}, \mathbf{s}}(f)$ is infinite  if the age vector in  state $\mathbf{s}$ is $(N, \cdots, N)$.
\end{remark}

\section{Proof of Theorem \ref{theorem:truncation} } \label{appendix:theorem:truncation}
 According to \cite[Theorem 8.6.6]{MDP:Puterman}, the RVIA in Eq. (\ref{eq:rvia}) converges to the optimal solution in finite iterations if the truncated MDP is \textit{unichain}, i.e., the Markov chain corresponding to every deterministic stationary policy consists of a single recurrent class plus a possibly empty set of transient states. Note that for every truncated MDP, there is  only one recurrent class by \cite{gallager}, since the state $(m,\cdots, m, 0, \cdots, 0)$ is reachable (e.g., there is no arrival in the next $m$ slots) from all other states (where remember that $m$ is the boundary of the truncated MDP). Hence, the truncated MDPs are unichain and the theorem follows immediately.

\section{Proof of Theorem \ref{theorem:stationary-omgea}}
\label{appendix:theorem:stationary-omgea} 
 Given  initial state $\mathbf{S}(0)=\mathbf{s}$, we define the expected total  $\alpha$-discounted cost  under  policy $\mu$ by
\begin{align*}
J_{\alpha}(\mathbf{s};\mu)=\limsup_{T \rightarrow \infty} E_{\mu}\left[ \sum_{t=0}^T \alpha^t C(\mathbf{S}(t), A(t))|\mathbf{s}(0)=\mathbf{S} \right],
\end{align*}
Let $J_{\alpha}(\mathbf{s})=\min_{\mu}J_{\alpha}(\mathbf{s};\mu)$ be the minimum   expected total $\alpha$-discounted cost. A policy that minimizes $J_{\alpha}(\mathbf{s};\mu)$ is called \textit{$\Omega_{\alpha}$-optimal policy}.  Again, we check the two conditions in Appendix~\ref{appendix:thm:stationary}. 
\begin{enumerate}
%

	\item   Let $f$ be the deterministic stationary policy of always choosing  action $A(t)=1$ for each slot $t$ if there is an arrival. It is obvious that the resulting DTMC by the policy is irreducible and aperiodic. To calculate the average cost, we note that age $X(t)$ by the policy $f$ is also a DTMC in Fig. \ref{fig:age-dtmc}. The steady-state distribution $\boldsymbol{\xi}=(\xi_1, \xi_2, \cdots, )$ of the DTMC is 
\begin{align*}
	\xi_i=p(1-p)^{i-1}\,\,\,\,\text{for all $i=1,2, \cdots$}.
\end{align*}	
Hence, the average age is 
\begin{align*}
\sum_{i=1}^{\infty} i \xi_i= \sum_{i=1}^{\infty} i p(1-p)^{i-1} = \frac{1}{p}.
\end{align*}
On the other hand, the average updating cost is $c \cdot p$ as the arrival probability is $p$. Hence,  the  average cost under the policy $f$ is   the average age (i.e., $1/p$) plus the average updating cost (i.e., $c\cdot p$), which is finite. 	

\begin{figure}[!t]
\centering
\includegraphics[width=.35\textwidth]{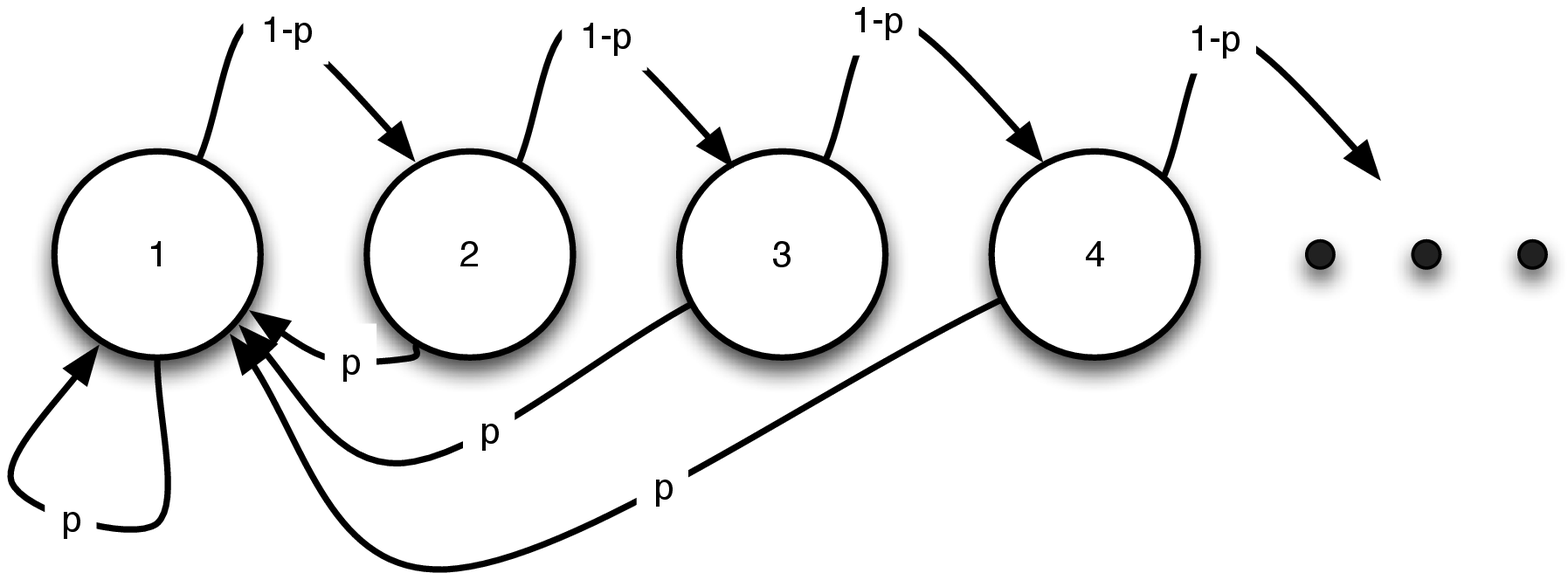}
\caption{The age $X(t)$ under the policy $f$ forms a DTMC.}
\label{fig:age-dtmc}
\end{figure}

\item Similar to  Proposition~\ref{lemma:monotone}, we can show that $J_{\alpha}(x,\lambda)$ is a non-decreasing function in  age $x$ for a given arrival indicator $\lambda$; moreover, $J_{\alpha}(x,\lambda)$ is a non-increasing function in $\lambda$ for a given age $x$. Thus, we can choose $L=0$. 
\end{enumerate}
By verifying the two conditions, the theorem immediately follows from \cite{stationary-policy:Sennott}.

\section{Proof of Lemma~\ref{lemma:optimal-threshold}}
\label{appendix:lemma:optimal-threshold}
Since $I(x,1)$ is the updating cost $c$ to make both actions for state $(x,1)$ equally desirable and we break a tie in favor of idling,  the optimal threshold is $x+1$ if the cost is $c=I(x,1)$, for all $x$. We  claim that  the optimal threshold monotonically increases with cost $c$, and then the theorem follows. 

To verify the claim,  we can focus on the discounted cost case according to the proof of Theorem \ref{theorem:optimal-switch}.  Suppose that an $\Omega_{\alpha}$-optimal action, associated with a cost $c_1$, for state $(x,1)$ is to idle, i.e., 
\begin{align*}
x+1+\alpha E[J_{\alpha}(x+1,\lambda')] \leq 1+c_1+\alpha E[J_{\alpha}(1,\lambda')] .
\end{align*}
Then, an $\Omega_{\alpha}$-optimal action, associated with a cost $c_2 \geq c_1$, for state $(x,1)$ is to idle as well since 
\begin{align*}
x+1+\alpha E[J_{\alpha}(x+1,\lambda')] \leq &1+c_1+\alpha E[J_{\alpha}(1,\lambda')] \\
\leq & 1+c_2+\alpha E[J_{\alpha}(1,\lambda')].
\end{align*}
Then, the monotonicity is established.

\section{Proof of Lemma \ref{lemma:index-optimal}}
\label{appendix:lemma:index-optimal}
Similar to the proof of Theorem \ref{theorem:optimal-switch}, we can focus on the discounted cost case. Without loss of generality, we assume that  age $x_1  \geq \max(x_2, \cdots x_N)$.

Let $\mathbf{x}_{ij}=(0, \cdots, x_j, \cdots, 0)$ be the zero vector except for the $i$-the entry being replaced by $x_j$. By the symmetry of the users, swap of the initial ages of any two users results in the same expected total  $\alpha
$-discounted cost, i.e., 
\begin{align*}
E[V_{\alpha}(x_1,\mathbf{x}_{-1}, \boldsymbol{\lambda})]=E[V_{\alpha}(x_j, \mathbf{x}_{-1}-\mathbf{x}_j+\mathbf{x}_{j1}, \boldsymbol{\lambda})],
\end{align*}
for all $j \neq 1$. Similar to the proof of Theorem \ref{theorem:optimal-switch}, here we focus on the case when $\lambda_1=1$ and $\lambda_j=1$. The result follows from the non-decreasing function of $V_{\alpha}(x_1,\mathbf{x}_{-1},\boldsymbol{\lambda})$ and $x_1 \geq x_j$ for all $j \neq 1$:
\begin{align*}
&\nu_{\alpha}(x_1, \mathbf{x}_{-1},\boldsymbol{\lambda};1)-\nu_{\alpha}(x_1, \mathbf{x}_{-1},\boldsymbol{\lambda};j)\\
=&x_j-x_1 +\alpha E[V_{\alpha}(1, \mathbf{x}_{-1}+\mathbf{1}, \boldsymbol{\lambda}')\\
& -V_{\alpha}(x_1+1, \mathbf{x}_{-1}+\mathbf{1}-\mathbf{x}_j, \boldsymbol{\lambda}')]\\
=&x_j-x_1 +\alpha E[V_{\alpha}(x_j+1, \mathbf{x}_{-1}+\mathbf{1}-\mathbf{x}_j, \boldsymbol{\lambda}')\\
&-V_{\alpha}(x_1+1, \mathbf{x}_{-1}+\mathbf{1}-\mathbf{x}_j, \boldsymbol{\lambda}')] \leq 0.\\
\end{align*}

\end{document}